\documentclass[11pt, a4paper]{amsart}
\usepackage{a4wide}
\usepackage{enumitem}
\usepackage{subfigure}
\usepackage{cancel}
\usepackage{stackengine}
\usepackage{amssymb}
\DeclareSymbolFontAlphabet{\amsmathbb}{AMSb}%

\usepackage[pdftex,
pdftitle={Approximation of the HEIDIH model},
bookmarksopen,
colorlinks,
linkcolor=black,
urlcolor=black,
citecolor=black
]{hyperref}
\hypersetup{pdfauthor={F. E. Benth, G. J. Lord, G. Di Nunno, A. Petersson}}

\usepackage{mathtools}

\newcommand{\IP}{\amsmathbb{P}}
\newcommand{\R}{\amsmathbb{R}}
\newcommand{\N}{\amsmathbb{N}}
\newcommand{\I}{\mathbb{1}} %
\newcommand{\Z}{\amsmathbb{Z}}

\newcommand{\cA}{\mathcal{A}}
\newcommand{\cC}{\mathcal{C}}
\newcommand{\cD}{\mathcal{D}} 
\newcommand{\cE}{\mathcal{E}}
\newcommand{\cF}{\mathcal{F}}
\newcommand{\cH}{\mathcal{H}}
\newcommand{\cJ}{\mathcal{J}}
\newcommand{\cL}{\mathcal{L}}
\newcommand{\cS}{\mathcal{S}}
\newcommand{\cU}{\mathcal{U}}

\DeclareMathOperator{\E}{\amsmathbb{E}} %
\DeclareMathOperator{\Cov}{\mathsf{Cov}}
\DeclareMathOperator{\trace}{Tr}

\newcommand{\dd}{\,\mathrm{d}}
\newcommand{\dom}{\mathrm{dom}} %

\newcommand{\bigE}[1]{\E\left[#1\right]}

\newcommand{\inpro}[3][{}]{ \langle #2 , #3 \rangle_{#1} }
\newcommand{\norm}[2][{}]{\| #2 \|_{#1}}
\newcommand{\lrnorm}[2][{}]{\left\| #2 \right\|_{#1}}
\newcommand{\Bignorm}[2][{}]{\Big\| #2 \Big\|_{#1}}

\newtheorem{lemma}{Lemma}[section]
\newtheorem{proposition}[lemma]{Proposition}

\newtheorem{theorem}[lemma]{Theorem}

\theoremstyle{remark}
\newtheorem{remark}[lemma]{Remark}

\theoremstyle{definition}

\newtheorem{assumption}[lemma]{Assumption}

\begin{document}
	\title[Approximation of the HEIDIH model]{The heat modulated infinite dimensional Heston model and its numerical approximation}
	
	\author[F.E.~Benth]{Fred Espen Benth} \address[Fred Espen Benth]{\newline Department of Mathematics, University of Oslo.
		\newline Postboks 1053, Blindern, 0316 Oslo, Norway.} \email[]{fredb@math.uio.no}
	
	\author[G.~Lord]{Gabriel Lord} \address[Gabriel Lord]{\newline Department of Mathematics, IMAPP, Radboud University.
		\newline Postbus 9010, 6500 GL Nijmegen, The Netherlands.} \email[]{gabriel.lord@ru.nl}
	
	\author[G.~Di~Nunno]{Giulia Di Nunno} \address[Giulia Di Nunno]{\newline Department of Mathematics, University of Oslo.
		\newline Postboks 1053, Blindern, 0316 Oslo, Norway.
		\newline and
		\newline Department of Business and Management Science, NHH Norwegian School of Economics.
		\newline Helleveien 30, 5045 Bergen, Norway.}
	\email[]{giulian@math.uio.no}
	
	\author[A.~Petersson]{Andreas Petersson} \address[Andreas Petersson]{\newline Department of Mathematics, University of Oslo.
		\newline Postboks 1053, Blindern, 0316 Oslo, Norway.} \email[]{andreep@math.uio.no}
	
	\thanks{\textit{Acknowledgement.} The research leading to these results was funded within the project STORM: Stochastics for Time-Space Risk Models, from the Research Council of Norway (RCN), project number: 274410.}
	
	
	\begin{abstract}
		The HEat modulated Infinite DImensional Heston (HEIDIH) model and its numerical approximation are introduced and analyzed. This model falls into the general framework of infinite dimensional Heston stochastic volatility models of (F.E. Benth, I.C. Simonsen '18), introduced for the pricing of forward contracts. The HEIDIH model consists of a one-dimensional stochastic advection equation coupled with a stochastic volatility process, defined as a Cholesky-type decomposition of the tensor product of a Hilbert-space valued Ornstein-Uhlenbeck process, the mild solution to the stochastic heat equation on the real half-line. The advection and heat equations are driven by independent space-time Gaussian processes which are white in time and colored in space, with the latter covariance structure expressed by two different kernels. First, a class of weight-stationary kernels are given, under which regularity results for the HEIDIH model in fractional Sobolev spaces are formulated. In particular, the class includes weighted Mat\'ern kernels. Second, numerical approximation of the model is considered. An error decomposition formula, pointwise in space and time, for a finite-difference scheme is proven. For a special case, essentially sharp convergence rates are obtained when this is combined with a fully discrete finite element approximation of the stochastic heat equation. The analysis takes into account a localization error, a pointwise-in-space finite element discretization error and an error stemming from the noise being sampled pointwise in space. The rates obtained in the analysis are higher than what would be obtained using a standard Sobolev embedding technique. Numerical simulations illustrate the results.
	\end{abstract}
	
	\maketitle
	
	\section{Introduction}
	
	This paper introduces the HEat modulated Infinite DImensional Heston (HEIDIH) model as a special case of the infinite dimensional stochastic volatility model of~\cite{BS18} and study its numerical approximation. We give a background to the class of models considered in~\cite{BS18} before discussing our results in detail.
	
	\subsection{Stochastic volatility models in infinite dimensions}
	In the last few years, infinite dimensional stochastic volatility models have garnered increasing interest from both analytical (e.g., in \cite{BS18,CT20,CKK22A,CKK22B}),  financial (in \cite{BK15,BDNS21}), as well as most recently (see~\cite{BE22,K23}) also from numerical perspectives. One of the motivations stems from modeling the risk-neutral dynamics of the forward price $f(t,T)$ of a contract at time $t \ge 0$ delivering a commodity (e.g., an instantaneous amount of energy) at time $T > t$. Under a Musiela parametrization $X(t,x) := f(t,x+t)$ the dynamics are interpreted as solving a stochastic partial differential equation (SPDE) on $\R^+$. The solution is formally given by
	\begin{equation}
		\label{eq:spde-approach}
		X(t,x) := (\cS(t) X(0))(x) + \int^t_0 (\cS(t-s) \sigma(s))(x) \dd B(s,x), \quad t, x \in \R^+ \times \R^+,
	\end{equation}
	with initial value $X(0) := f(0,\cdot)$. Here $\cS$ is the left-shift semigroup on functions on $\R^+$, $(B(\cdot,x))_{x \ge 0}$ is a family of Brownian motions and $\sigma$ an appropriate stochastic volatility term. This is made rigorous by interpreting~\eqref{eq:spde-approach} as the mild solution of a stochastic evolution equation in a Hilbert space $\cH$ (in the Da Prato--Zabczyk framework \cite{DPZ14}) of continuous and eventually constant functions on $\R^+$, driven by a Wiener process $B$ in $\cH$. We emphasize that~\eqref{eq:spde-approach} models the \textit{risk-neutral} dynamics of forward prices. For the actual market dynamics one would need to include additional terms, such as a seasonality dependent drift. We refer to \cite{BK14,BK15} for more details on the infinite dimensional approach to forward price dynamics in commodity markets, as well as to \cite{HJM92,F01} for predecessor models in interest rate markets.
	
	Once the model~\eqref{eq:spde-approach} has been completely specified (by choosing $B$ and $\sigma$ appropriately), a natural next step is to analyze the numerical approximation of the forward problem. By this we mean the problem of simulating approximate realizations of~\eqref{eq:spde-approach} and quantifying the uncertainty stemming from the numerical approximation necessary in any infinite-dimensional model. This is important, not only for Monte Carlo approaches to pricing derivatives written on $f$, but also to efficiently solve inverse problems stemming from real-world data using approaches that rely on simulations of~\eqref{eq:spde-approach}, such as machine learning algorithms.
	
	\subsection{The HEIDIH model}
	In this paper, we present, analyze and study the numerical approximation of the HEIDIH model. It is a special case of the infinite dimensional stochastic volatility model introduced in~\cite{BS18,BDNS21}, in turn a special case of~\eqref{eq:spde-approach}. The price process $X$ and a volatility process $Y$ are coupled via another process $Z$ in the system of stochastic evolution equations
	\begin{equation}
		\label{eq:spde-system}
		\begin{dcases*}
			\dd X(t) = \cC X(t) \dd t + \Gamma^Z(t) \dd B(t), \quad X(0) \in \cH, \\
			\dd Y(t) = \cA Y(t) \dd t + \dd W(t), \quad Y(0) \in \cH,
		\end{dcases*}
	\end{equation}
	for $t \in [0,T], T < \infty$. Here $(\cH, \inpro[\cH]{\cdot}{\cdot})$ a Hilbert space of functions on $\R^+$ on which the semigroup $\cS$ is strongly continuous. With $\cU$ being another strongly continuous semigroup, $\cC := \partial/\partial x$ and $\cA$ denote the generators of $\cS$ and $\cU$, respectively. By $B$ and $W$ we denote two independent cylindrical Wiener processes in $\cH$. 
	The operator-valued mapping $t \mapsto \Gamma^Z(t)$ corresponds to $\sigma$ in~\eqref{eq:spde-approach} and is given by $\Gamma^Z(t) := Y(t) \otimes Z(t)$. Here, the stochastic process $Z$ takes values on the unit ball in $\cH$ and $\otimes$ denotes the Hilbert tensor product. The process $Z$ is an intrinsic part of the model and each choice yields different dynamics. One example is the setting $Z(t) := \eta$ for all $t \in [0,T]$, where $\eta$ is a deterministic function with unit norm. Another is obtained by letting $Z(t) := Y(t)/\norm[\cH]{Y(t)}$ for $Y(t) \neq 0$ and $1$ otherwise. In any case $\Gamma^Z(t) (\Gamma^Z(t))^* = Y(t) \otimes Y(t)$, where $(\Gamma^Z(t))^*$ is the adjoint of $\Gamma^Z(t)$, so that $\Gamma^Z(t)$ is a square root of the stochastic variance process $Y(t) \otimes Y(t)$ \cite[Proposition~7]{BS18}. This is why the system is referred to as a Heston stochastic volatility model in Hilbert space.
	
	The HEIDIH model is obtained by specifying $\cA = a\Delta$ in~\eqref{eq:spde-system}, which is the Laplacian $\Delta$ with either Dirichlet or Neumann boundary conditions at $x = 0$ coupled with a diffusivity constant $a > 0$. Then $Y$ is the solution to a stochastic heat equation. For the state space $\cH$ we adopt the setting of~\cite{ET07} and choose $\cH := \cH^r := H^r(\R^+) \oplus \R$ where $H^r(\R^+)$ is a fractional Sobolev space and $r > 1/2$.  Then $B$ and $W$ may be regarded as random field-valued processes. An example of a realization of $Y$ and $X$ for this model when $Z(t) := \eta$ for all $t \in [0,T]$ and fixed $\eta \in \cH$ is shown in Figures~\ref{subfig:Y} and~\ref{subfig:X}, see Remark~\ref{rem:example-figure} for precise parameter choices. 
	
	\begin{figure}[ht!]
		\centering
		\subfigure[Volatility process (stochastic heat equation) $Y = (Y(t,x))_{(t,x)\in [0,T] \times [0,2T]}$. \label{subfig:Y}]{\includegraphics[width = 0.99\textwidth]{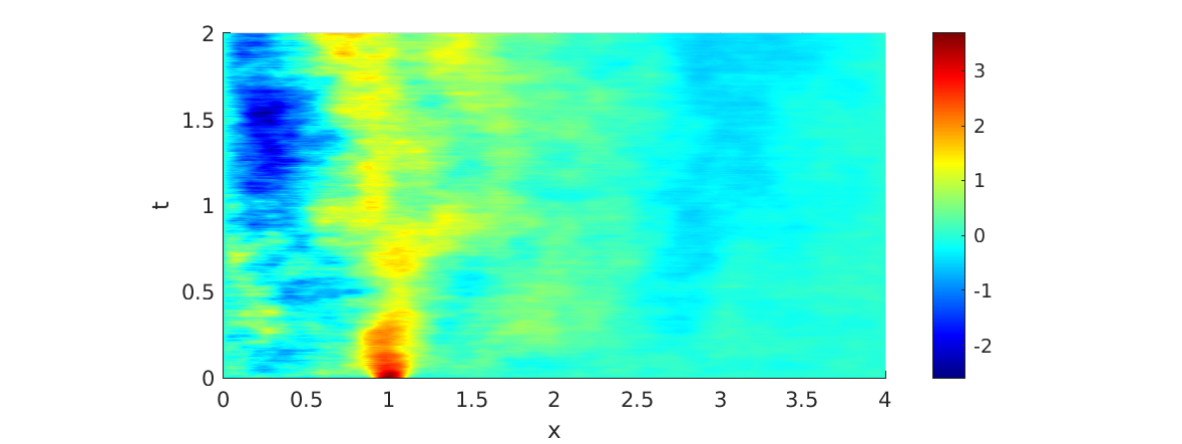}}
		\subfigure[Price process $X = (X(t,x))_{(t,x)\in [0,T]^2}$.\label{subfig:X}]{\includegraphics[width = .49\textwidth]{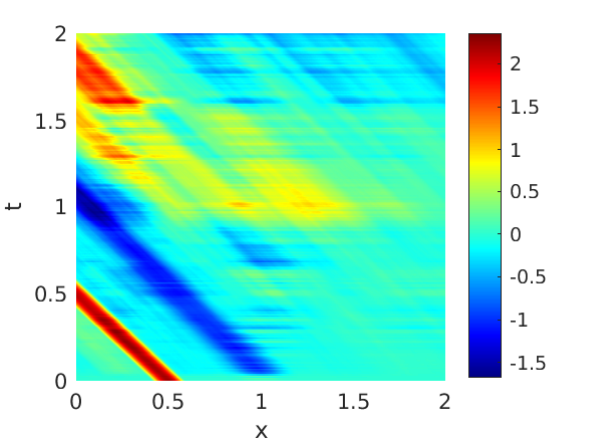}}
		\subfigure[Covariance function $(c_T(x,y))_{(x,y)\in [0,2T]^2}$, where $c_T(x,y) = \Cov(Y(T,x),Y(T,y))$. \label{subfig:cov}]{\includegraphics[width = .49\textwidth]{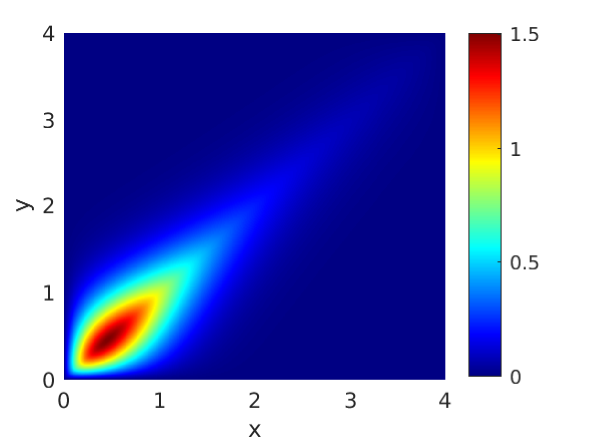}}
		\caption{Realization of the HEIDIH model $X$ in the case that $Z(t) = \eta \in \cH$ for $t \in [0,T] = [0,2]$, the corresponding realization of $Y$ and an illustration of the covariance function $(x,y) \mapsto \Cov(Y(T,x), Y(T,y))$. The process $Y$ has zero Dirichlet boundary conditions.}
		\label{fig:example}
	\end{figure}
	
	We emphasize that to analyze the numerical approximation of~\eqref{eq:spde-system} in detail, both the generator $\cA$ and the Hilbert space $\cH$ have to be completely specified. We believe the HEIDIH specification may be appropriate since 
	\begin{enumerate}[label=(\roman*)]
		\item the model directly displays the Samuelson effect. This says that the volatility should be decreasing in time to maturity (see, e.g., \cite{BBK08} for a discussion of this in a commodity pricing context) i.e., that $Y(t,x) \to 0$ as the time to maturity $x \to 0$. This is obtained from the fact that $Y(t) \in H^r(\R^+) \subset \cH^r$, see also Figure~\ref{subfig:Y}. Moreover,
		\item by properties of the heat kernel,  $\Cov(Y(s,x),Y(s,y)) \to 0$ as $|x-y| \to \infty$ (see Figure~\ref{subfig:cov}). The left-shift property of the semigroup $\cS$ implies that  
		\begin{equation*}
			X(t,x) = X(0,x + t) + \int^t_0 Y(s,x+s) \inpro[\cH]{Z(s)}{\cdot} \dd B(s).
		\end{equation*}
		The dependence of the volatility of $X$ with regards to time to maturity is therefore local. This is realistic, since contracts close in maturity affect each other. Furthermore,
		\item freedom is still left to the practitioner to specify the model further given the specific context at hand. This is done in a parsimonious way by appropriately choosing the coefficient $a$ (which regulates the degree to which the dependence of the volatility with regards to time to maturity is local) the initial functions $X(0)$ and $Y(0)$ and the dynamics of $B$ and $W$ in terms of their covariance kernels. Finally, 
		\item precise regularity results and approximation convergence rates can be obtained under realistic conditions on the covariance kernels of $B$ and $W$, which is vital to designing an efficient simulation algorithm.
	\end{enumerate}
	
	The latter point is the main focus of our paper, which contains, to the best of our knowledge, \textit{the first analysis of the numerical approximation} of the infinite dimensional Heston model introduced in~\cite{BS18,BDNS21}. Note that the model may easily be modified by adding additional terms to either of the equations in~\eqref{eq:spde-system} without violating the validity of the points above. For instance, one may add a deterministic drift term to $Y$ corresponding to seasonal volatility, which might be present even under the risk-neutral measure. 
	
	\subsection{Contributions of this paper}
	Having introduced the model, let us now enumerate the main contributions of this paper.
	
	In the context of forward pricing, the properties of $B$ are often discussed in terms of its (incremental) covariance kernel $q(x,y) = \Cov(B(1,x),B(1,y))$. At the same time, it is common to assume that $B$ is a $Q$-Wiener process, see, e.g., \cite{BB14,BS18,BDNS21}. It is, however, rare to find any discussion on what properties a given kernel $q$ should have in order for $B$ (or $W$) to be well-defined in $\cH$ either as a cylindrical or as a $Q$-Wiener process, thus allowing for the existence of $X(t)$ in a given Hilbert space $\cH$. Thanks to our particular choice of state space $\cH$ from~\cite{ET07}, we are able to 
	\begin{enumerate}[label=(\roman*)]
		\item \textit{completely characterize} the kernels that yield classical $Q$-Wiener processes in $\cH$ (Theorem~\ref{thm:wiener-covariance-functions}) and
		\item \textit{construct a class of kernels}, which we call weight-stationary kernels, that allow for precise regularity and sharp numerical convergence results for $X$ and $Y$ (Theorem~\ref{thm:noise-kernel}).
	\end{enumerate}
	The constructed class includes weighted versions of kernels commonly encountered in applications, such as Mat\'ern kernels.
	We believe these general ideas are of interest also for infinite-dimensional models outside the context of forward price modeling in commodity markets. Note that no results of this nature are present in~\cite{ET07}. Since the underlying spatial domain $\R^+$ is unbounded, non-trivial arguments are needed to derive these theorems, including the use of results from~\cite{HT94} on the decay of entropy numbers of embeddings between weighted Besov-type spaces. 
	
	When analyzing a numerical approximation of the HEIDIH model, the most natural errors to analyze are pointwise in space, since $X(t,\cdot)$ evaluated at $x = T-t$ corresponds to the forward price of a contract with maturity time $T$. Like most algorithms in the numerical analysis literature on SPDEs, the convergence rates of our particular approximation depend on the regularity of the parts making up the model, in particular on $B$ and $W$. Whereas Sobolev regularity (as well as existence and uniqueness) of $Y$ and $X$ follow from classical results of~\cite{DPZ14} and~\cite{JR12} when weight-stationary kernels are considered for $B$ and $W$, it turns out that this is not enough for pointwise-in-space. Instead, for sharp numerical results, we have to derive
	\begin{enumerate}[label=(\roman*), resume]
		\item \textit{spatiotemporal regularity} estimates of $Y$ in a H\"older sense (Proposition~\ref{prop:holder}) applicable to weight-stationary kernels.
	\end{enumerate}
	These estimates cannot be obtained from classical results via a Sobolev embedding technique.
	
	Having studied the regularity of the model, the component $X$ is approximated by means of a fully discrete finite difference scheme. We supply
	\begin{enumerate}[label=(\roman*), resume]
		\item an \textit{error decomposition formula for the approximation of $X$}  (Proposition~\ref{prop:X-error-decomp}), applicable to general infinite dimensional Heston models in the sense of~\cite{BS18}. 
	\end{enumerate}
	The resulting error depends on how well the process $Y$ can be approximated. We then focus on approximation of $Y$ in the special case of Dirichlet boundary conditions and $Z(t) = \eta$. A backward Euler discretization in time is combined with a finite element approximation in space. While this particular discretization of the stochastic heat equation is well-studied (see Section~\ref{sec:conclusion} for a comparison to existing results), the situation in our case stands out for three reasons. First, we need to take into account a truncation of our unbounded domain $\R^+$, resulting in localization errors. Second, we consider pointwise-in-space errors, which means we cannot rely on standard estimates for the finite element method. Third, we need to simulate the noise $W$ efficiently using the circulant embedding method, which results in an additional source of errors, since the noise is interpolated rather than projected onto the finite element space. With this in mind, we
	\begin{enumerate}[label=(\roman*), resume]
		\item \textit{derive sharp error convergence estimates for the approximation of $Y$}  (Proposition~\ref{prop:fem-convergence}) which shows how the error depends on localization, SPDE discretization and noise interpolation error sources. 
	\end{enumerate}
	The sharpness is demonstrated in simulations (Figures~\ref{fig:spatially-semidiscrete-errors} and~\ref{fig:temporally-semidiscrete-errors}). Combining Propositions~\ref{prop:X-error-decomp} and~\ref{prop:fem-convergence}, we then finally
	\begin{enumerate}[label=(\roman*), resume]
		\item obtain a \textit{full error estimate for the approximation of $X$}  (Theorem~\ref{thm:full-approximation}).
	\end{enumerate}
	This estimate states how the error for the approximation of $X$ depends on the quality of the approximation of $Y$. Not only is this relevant from a mathematical point of view, but has practical implications too. The explicit dependence of the error on the different sources outlined above informs practitioners how the computational effort should be divided so that each source contributes equally. This leads to a computationally efficient algorithm. We demonstrate this point in simulations using different choices of spatial and temporal step sizes for $Y$ (Figures~\ref{fig:temporalerror-comparison-plot} and~\ref{fig:comparison-times}). Our earlier regularity theory reveals how the different step sizes should be tuned for computational efficiency, connecting theory to practice.
	
	\subsection{Outline}
	
	We end this section with a brief outline of this paper. In Section~\ref{sec:preliminaries} we introduce the necessary mathematical background and notation. Section~\ref{sec:wiener-processes} contains our results on cylindrical Wiener processes in $\cH$, followed by the regularity theory for the HEIDIH model in Section~\ref{sec:regularity}. Our numerical analysis along with accompanying simulations are contained in Section~\ref{sec:numerical}. In Section~\ref{sec:conclusion} we discuss how our results relate to the existing literature for numerical approximations of SPDEs and outline future work.

	\section{Preliminaries}
	\label{sec:preliminaries}
	
	This section briefly presents the mathematical machinery necessary for our results.
	
	\subsection{Notation and operator theory}
	
	Let $(U, \norm[U]{\cdot})$ and $(V, \norm[V]{\cdot})$ be Banach spaces. All Banach spaces in this paper are taken over $\R$ unless otherwise stated. We denote by $\cL(U,V)$ the space of bounded linear operators from $U$ to $V$ equipped with the usual operator norm. If $U$ and $V$ are separable Hilbert spaces we write $\cL_2(U,V)$ for the space  Hilbert--Schmidt operators. This is a separable Hilbert space with an inner product, for an arbitrary ONB (orthonormal basis) $(e_j)_{j = 1}^\infty$ of $U$, given by
	$$\inpro[\cL_2(U,V)]{\Gamma_1}{\Gamma_2} := \sum_{j = 1}^{\infty} \inpro[V]{\Gamma_1 e_j}{\Gamma_2 e_j}\text{ for }\Gamma_1, \Gamma_2 \in \cL_2(U,V).$$
	We have $\Gamma \in \cL_2(U,V)$ if and only if $\Gamma^* \in \cL_2(V,U)$ and their norms coincide. We use the shorthand notations $\cL(U)$ and $\cL_2(U)$ when $V=U$ and we write $\Sigma^+(U)$ for the class of positive semidefinite operators when $U$ is a Hilbert space. For operators $\Gamma \in \Sigma^+(U)$, we say that they are trace class if $\trace(\Gamma) : = \sum^\infty_{j=1} \inpro[]{\Gamma e_j}{e_j} = \norm[\cL_2(U)]{\Gamma^{1/2}}^2 < \infty$ for one, equivalently all, ONBs $(e_j)_{j = 1}^\infty$ of $U$. 
	
	For Hilbert spaces $U$ and $V$, the tensor $u \otimes v$ is regarded as an element of $\cL(V,U)$ by the relation $(u \otimes v) w := \inpro[V]{v}{w} u$ for $v,w \in V$ and $u \in U$.
	
	We need the concepts of approximation and entropy numbers from~\cite{P80}. Both types of numbers are used to derive suitable properties of a class of covariance kernels in Theorem~\ref{thm:noise-kernel}. Approximation numbers also comes into the numerical analysis of the stochastic heat equation approximation in Proposition~\ref{prop:fem-convergence}. Given a linear operator $\Gamma$ between Banach spaces $U$ and $V$, its $n$th approximation number $\psi_n(\Gamma)$ is given by 
	\begin{equation*}
		\psi_n(\Gamma) := \inf\Big\{ \norm{\Gamma - \tilde \Gamma} : \tilde \Gamma \in \cL(U,V),  \text{rank}(\tilde \Gamma) < n \Big\},
	\end{equation*}
	$n \in \N$, while its $n$th entropy number $\epsilon_n(\Gamma)$ is 
	\begin{equation*}
		\epsilon_n(\Gamma) := \inf\Big\{ \delta \ge 0 : \Gamma(B_U) \subset \cup_{j = 1}^m \{v_j + \delta B_V \} \text{ for some } v_1, \ldots, v_m \in V \text{ and } m \le 2^{n-1} \Big\}.
	\end{equation*}
	Here $B_U$ and $B_V$ denote the unit spheres in $U$ and $V$, respectively. By 
	\cite[Theorems~11.2.3, 12.1.3]{P80}, both types of numbers satisfy
	\begin{equation}
		\label{eq:operator-norm-property}
		\epsilon_{1}(\Gamma) = \psi_1(\Gamma) = \norm[\cL(U,V)]{\Gamma}.
	\end{equation}
	Both numbers are also multiplicative in that, if $\Gamma_1 \in \cL(U,V)$ and $\Gamma_2 \in \cL(\tilde V,V)$ for some additional Banach space $\tilde V$, then 
	\begin{equation}
		\label{eq:multiplicative-property}
		\epsilon_{m+n-1}(\Gamma_1 \Gamma_2) \le \epsilon_m(\Gamma_1) \epsilon_n(\Gamma_2) \text{ and } \psi_{m+n-1}(\Gamma_1 \Gamma_2) \le \psi_m(\Gamma_1) \psi_n(\Gamma_2)
	\end{equation}
	for all $m,n \in \N$ \cite[Theorems~11.9.2 and~12.1.5]{P80}. 
	If $U$ and $V$ are Hilbert spaces, then
	\begin{equation}
		\label{eq:approximation-bound-by-entropy}
		\psi_n(\Gamma) \le 2 \epsilon_n(\Gamma) 
	\end{equation}
	for all $n \in \N$ and
	\begin{equation}
		\label{eq:hs-norm-approximation-numbers}
		\norm[\cL_2(U,V)]{\Gamma} = \norm[\ell^2]{(\psi_n(\Gamma))_{n=1}^\infty},
	\end{equation}
	where $\ell^p$ denotes the usual space of $p$-summable sequences \cite[Theorems~11.3.4, 12.3.1 and 15.5.5]{P80}. Taking~\eqref{eq:operator-norm-property} and~\eqref{eq:multiplicative-property} into account, this means in particular that $\cL_2(U,V)$ is an operator ideal. 
	
	For two Banach spaces $U$ and $V$, we write $U \hookrightarrow V$ if $U \subset V$ and $\norm[V]{u} \le C \norm[U]{u}$ for some constant $C < \infty$ and all $u \in U$, i.e., the embedding operator $I_{U \hookrightarrow V} \in \cL(U,V)$. For Hilbert spaces, we write $U  \xhookrightarrow[]{\cL_2} V$ as shorthand for $I_{U \hookrightarrow V} \in \cL_2(U,V)$
	
	Throughout this paper, we adopt the notion of generic constants, which may vary from occurrence to occurrence and are independent of any parameter of interest, such as spatial or temporal step sizes. By $a \lesssim b$, for $a,b \in \R$, we denote the existence of a generic constant such that $a \le C b$. 
	
	\subsection{Reproducing kernel Hilbert spaces}
	
	We make heavy use of the theory of reproducing kernel Hilbert spaces (RKHSs). The properties that we are going to need are listed here, we refer to, e.g., \cite{BT04, W04} for further details. Throughout the paper we consider symmetric positive semidefinite kernels $q$ on a non-empty subset $\cE \subseteq \R$ but usually just refer to $q$ as a kernel on $\cE$ when there is no risk of confusion. A Hilbert space $H = H_q(\cE)$ is said to be the RKHS of a kernel $q$ if it is a Hilbert space of real-valued functions on a non-empty index set $\cE \subseteq \R$ such that the conditions 
	\begin{enumerate}[label=(\roman*)]
		\item \label{RKHSprop1} $q(x,\cdot) \in H_q(\cE)$ for all $x \in \cE$ and
		\item \label{RKHSprop2} for all $f \in H_q(\cE), x \in \cE$, $f(x)=\inpro[H_q(\cE)]{f}{q(x,\cdot)}$
	\end{enumerate}
	are satisfied. The property~\ref{RKHSprop2} is referred to as the reproducing kernel property of the space $H_q(\cE)$. For each kernel $q$ there is one and only one Hilbert space of functions on $\cE$ with these two properties. Moreover, a Hilbert space $H$ of functions on $\cE$ is a RKHS if and only if the evaluation functional $\delta_x \colon H \to \R$, defined by $\delta_x f = f(x)$, is continuous for all $x \in \cE$ \cite[Theorem~10.2]{W04}. We write $H_q$ for $H_q(\cE)$ when it is clear from the context what index set $\cE$ we have in mind. Since $q$ is positive semidefinite, $|q(x,y)|^2 \le q(x,x) q(y,y)$ for all $x,y \in \cE$. 
	
	If $H_q$ is separable with an ONB $(e_j)_{j=1}^\infty$,  \cite[Theorem~14]{BT04} yields the kernel decomposition $q(x,\cdot) = \sum^\infty_{j=1} e_j(x)e_j$, the sum being convergent in $H_q$. This implies that $q(x,y) = \sum^\infty_{j=1} e_j(x)e_j(y)$ with convergence in $\R$. For index sets $\cE$ that are domains in $\R$, continuity of $q$ is a sufficient condition for separability \cite[Theorem~15]{BT04}.
	
	Suppose that we are given two RKHSs on an index set $\cE \subseteq \R$ with kernels $q$ and $\tilde q$. Then $H_{\tilde q} \hookrightarrow H_q$ if and only if there is some constant $C \ge 0$ such that $C q - \tilde q$ is a positive semidefinite kernel \cite[Theorem~I.13.IV, Corollary~I.13.IV.2]{A50}. We express this by writing $q \gtrsim \tilde q$. 
	
	For any Hilbert spaces $U, H$ and operator $\Gamma \in \cL(U,H)$, we may interpret the range $\Gamma(U)$ as a Hilbert space when equipped with the inner product $\inpro[\Gamma(U)]{\cdot}{\cdot} := \inpro[U]{\Gamma^{-1}\cdot}{\Gamma^{-1}\cdot}$, where $\Gamma^{-1}$ is the pseudoinverse of $\Gamma$ \cite[Proposition~C.0.3]{PR07}. The norm on $\Gamma(U)$ may also be represented by 
	\begin{equation*}
		\norm[\Gamma(U)]{u} = \min_{\substack{v \in U \\ \Gamma v = u}} \norm[U]{v},
	\end{equation*}
	see \cite[Remark~C.0.2]{PR07}. Applying this to a RKHS $H_q(\cE)$ and the operator $R_{\cE \to \cD}$ that restricts functions on $\cE$ to functions on a subset $\cD \subset \cE$, it follows from \cite[Theorem~6]{BT04} that $R_{\cE \to \cD}$ maps $H_q(\cE)$ to $H_q(\cD)$ and $H_q(\cD) = R_{\cE \to \cD}(H_q(\cE))$ with equal norms. Moreover, by~\cite[Theorem~10.46]{W04} there is a linear extension operator $E_{\cD \to \cE}$ of functions on $\cD$ to functions on $\cE$ such that $\norm[\cL(H_q(\cD),H_q(\cE))]{E_{\cD \to \cE}} = 1$. %
	
	\subsection{Fractional Sobolev spaces}
	
	For a possibly unbounded $\cE \subset \R$, we write $L^p(\cE)$, $p \in [0,\infty]$, for the usual Banach space of $p$-integrable equivalence classes of functions with respect to the Lebesgue measure. We omit $\cE$ when it is clear from the context which domain we refer to. The fractional Sobolev space $H^r(\R)$, $r \ge 0$, is a separable Hilbert space and a generalization of the usual Sobolev space $H^k(\R) = W^{k,2}(\R)$, $k \in \N$. It consists of all $u \in L^2(\R)$ such that 
	\begin{equation*}
		\norm[H^r(\R)]{u}^2 := (2 \pi)^{-1/2} \int_{\R} |\hat{u}(\xi)|^2 (1 + |\xi|^2)^r \dd \xi < \infty,
	\end{equation*} 
	where $\hat{u}$ is the Fourier transform of $u \in L^2(\R)$. For $\cE \subset \R$, the fractional Sobolev space $H^r(\cE)$, $r \ge 0$, consists of all $u \in L^2(\cE)$ for which $u = \tilde u|_{\cE}$ for some $\tilde u \in H^r(\R)$. It is equipped with the norm
	\begin{equation*}
		\norm[H^r(\cE)]{u} := \min_{\substack{v \in H^r(\R) \\ v|_{\cE}=u}} \norm[H^r(\R)]{v}.
	\end{equation*} 
	Since $\cE$ is an interval, $H^r(\cE)$ has a bounded extension operator. It therefore coincides with the usual Sobolev space $W^{r,2}(\cE)$ with equivalent norms. We refer to \cite[Chapter~VI]{S70} for more details on this. In terms of the restriction operator $R_{\R \to \cE}$, we have $H^r(\cE) = R_{\R \to \cE}(H^r(\R))$ with equal norms. We again omit $\cE$ from $H^r(\cE)$ when it is clear from the context which domain we refer to. 
	
	If $r > 1/2$, then $H^r(\cE)$, $\cE \subseteq \R$, is a RKHS with a stationary (that is, $m_r(x,y) = m_r(x-y)$ for $x,y \in \R$) kernel given by 
	\begin{equation*}
		m_r(x) := 2^{1 - r}(\Gamma(r))^{-1} |x|^{r-1/2} K_{r-1/2}(|x|),
	\end{equation*}
	see \cite[Chapter~10]{W04}. Here $K_{r-1/2}$ denotes the modified Bessel function of the second kind and order $r-1/2$. Note that $m_r$ is continuous and bounded on $\R$. By the reproducing property, we obtain that all functions $f \in H^r(\cE)$ are continuous and bounded, which is a special case of the classical Sobolev embedding theorem.
	
	We also make use of the weighted spaces $H^{r,\alpha} := H^{r,\alpha}(\R)$ where $r, \alpha \ge 0$. For these we introduce the weight function $w_\alpha$ given by $w_\alpha(x) := (1 + x^2)^{\alpha}$ for $x, \alpha \in \R$. We set $H^{r,\alpha}(\R) := \{u \in L^2(\R) : x \mapsto w_\alpha(x) u(x) \in H^r(\R)\}$. These are special cases of the weighted spaces considered in~\cite{HT94}, which we refer to for further details, see also~\cite[Theorems~2.3.9 and~2.5.6]{T83}. They are Hilbert spaces equipped with the inner product $\inpro[H^{r,\alpha}]{u}{v} := \inpro[H^{r}]{w_\alpha u}{w_\alpha v}$. For $s \ge r, \beta \ge \alpha$, $H^{s,\beta} \hookrightarrow H^{r,\alpha}$ with the embedding being compact if and only if both inequalities are strict and $\beta > 0$ \cite[Theorem~2.3]{HT94}. For $r > 1/2$, we note that $w_{-\alpha}(x) m_r(x-\cdot) w_{-\alpha}(\cdot) \in H^{r,\alpha}$ for all $x \in \R$ and that for $f \in H^{r,\alpha}$, 
	\begin{equation*}
		w_\alpha(x) f(x) = \inpro[H^r]{w_\alpha f}{m_r(x-\cdot)} = \inpro[H^{r,\alpha}]{f}{w_{-\alpha} m_r(x-\cdot)}, \, x \in \R.
	\end{equation*}
	From this, we obtain that $H^{r,\alpha}$ is a RKHS with kernel $(x,y) \mapsto w_{-\alpha}(x) m_r(x-y) w_{-\alpha}(y)$.
	
	The state space that we will work with for the HEIDIH model is given by $\cH^r := H^r(\R^+) \oplus \R$, $r > 1/2$. We define it for $r \ge 0$. It consists of all $f = \tilde f + b$ where $\tilde f \in H^r(\R^+)$ and $b \in \R$. It is equipped with the norm 
	\begin{equation*}
		\norm[\cH^r]{f}^2 := \norm[H^r(\R^+)]{\tilde f}^2 + b^2.
	\end{equation*}
	When $r > 1/2$, it is a RKHS with a kernel given by $(x,y) \mapsto 1 + m_r(x-y)$. We consider $H^r$ as a subspace of $\cH^r$ by identifying it with $H^r + 0$.
	
	\section{Cylindrical Wiener processes in fractional Sobolev spaces}
	\label{sec:wiener-processes}
	
	In this section, we recall the concept of cylindrical Wiener processes with a focus on processes that have a covariance kernel. We prove a result that completely characterizes $Q$-Wiener processes in the state space $\cH^r$. Then, we outline the assumptions on the kernels $q_W$ and $q_B$ that are made in the forthcoming sections. Finally, we construct the class of weight-stationary kernels, that fulfill all the assumptions made.
	
	Consider $T < \infty$, a filtered probability space $(\Omega,\cF,(\cF_t)_{t \in [0,T]},\IP)$ fulfilling the usual conditions and a Hilbert space $H$. 
	We follow~\cite{R11} and define a (strongly) cylindrical Wiener process as a process $W \colon [0,T] \to \cL(H,L^2(\Omega,\R))$ such that $t \mapsto W(t)u$ is a real-valued Wiener process for all $u \in H$. This definition agrees with what is called a generalized Wiener process in~\cite{DPZ14}. 
	
	Given a separable RKHS $H_q \hookrightarrow H$ with ONB $(e_j)_{j=1}^\infty$ and a sequence $(\beta_j)_{j=1}^\infty$ of real-valued Wiener processes, we construct a cylindrical Wiener process $W$ in $H$ by
	\begin{equation}
		\label{eq:wiener-construction}
		W(t)u := \sum^\infty_{j=1} \beta_j(t) \inpro[H]{u}{e_j}, \, t \in [0,T], \, u \in H,
	\end{equation}
	which converges in $L^2(\Omega,\R)$, see~\cite[Remark~7.3]{R11}. Note that $\beta_j = W(\cdot)e_j$ for $j \in \N$. The representation does not depend on $(e_j)_{j=1}^\infty$ in the sense that if $(f_j)_{j=1}^\infty$ is another ONB of $H_q$ and we define a real-valued Wiener processes $\tilde \beta$ by $\tilde \beta_j = W(\cdot)f_j$, then 
	\begin{equation*}
		W(t)u = \sum^\infty_{j=1} \tilde \beta_j(t) \inpro[H]{u}{f_j}, \, t \in [0,T], \, u \in H, \, \IP\text{-a.s.}
	\end{equation*}
	If $U$ is another Hilbert space such that $H_q \hookrightarrow U$, we obtain a cylindrical Wiener process in  $U$ by replacing $\inpro[H]{u}{e_j}$ in~\eqref{eq:wiener-construction} with $\inpro[U]{u}{e_j}$. In $H$, $W$ will have an incremental covariance operator $Q \in \Sigma^+(H)$ given by $Q = I_{H_q \hookrightarrow H} I^*_{H_q \hookrightarrow H}$ in the sense that $\E[(W(1)u)(W(1)v)] = \inpro[H]{Q u}{v}$ for all $u,v \in H$. Similarly, $W$ has covariance $I_{H_q \hookrightarrow U} I^*_{H_q \hookrightarrow U}$ in $U$. These operators have unique positive square roots. By~\cite[Corollary~B.3]{DPZ14}, $(I_{H_q \hookrightarrow H} I^*_{H_q \hookrightarrow H})^{1/2}(H) = (I_{H_q \hookrightarrow U} I^*_{H_q \hookrightarrow U})^{1/2}(U) = H_q$. In this sense, the distribution of $W$ does not depend on the choice of $H$ in~\eqref{eq:wiener-construction}. Therefore, we are justified in calling $W$ a cylindrical Wiener process with kernel $q$ and need not specify in which Hilbert space we consider it. 
	
	The SPDEs we consider in the paper are built on stochastic integrals with respect to cylindrical Wiener processes $W$. Consider a predictable process $\Psi \colon [0,T] \to \cL_2(H_q,H)$. It\^o integrals taking values in $H$ are well-defined with 
	\begin{equation*}
		\E\left[\lrnorm[H]{\int_{0}^{t} \Psi(s) \dd W(s)}^2\right] = \int_{0}^{t} \E[\norm[\cL_2(H_q,H)]{\Psi(s)}^2] \dd s,
	\end{equation*}
	for $t \in [0,T]$, provided that the integral on the right hand side is finite. We refer to \cite{R11} and \cite[Section~4.2]{DPZ14} for more details on Wiener processes in Hilbert spaces and the It\^o integral.
	
	If $H_q \xhookrightarrow[]{\cL_2} H$, the sum 
	\begin{equation*}
		W(t) := \sum^\infty_{j=1} \beta_j(t) e_j, \, t \in [0,T],
	\end{equation*}
	converges in $L^2(\Omega,H)$. It is then called a $Q$-Wiener process, and its covariance operator $Q = I_{H_q \hookrightarrow H} I^*_{H_q \hookrightarrow H}$ is trace-class. It induces a cylindrical Wiener process by $W(t) u := \inpro{W(t)}{u}$ and we do not make a notational difference between the two concepts. 
	Many papers dealing with SPDE models for forward prices assume that $W$ is a $Q$-Wiener process, see, e.g., \cite{BDNS21,BS18}. It is therefore important to clarify when this is the case in our setting with the state space $\cH^r$. The following theorem completely characterizes the kernels $q$ that satisfy $H_q \xhookrightarrow[]{\cL_2} \cH^r$, so that $W$ is a $Q$-Wiener process in $\cH^r$.
	
	\begin{theorem}
		\label{thm:wiener-covariance-functions}
		There is a separable RKHS $H_q = H_q(\R^+)$ such that $H_q \xhookrightarrow[]{\cL_2} \cH^r$ if and only if there is a separable RKHS $H_{\tilde q} = H_{\tilde q}(\R^+)$ such that $H_{\tilde q} \xhookrightarrow[]{\cL_2} H^r$, a function $f \in H_{\tilde q}$ and a constant $c \ge 0$ with
		\begin{equation*}
			q(x,y) = \inpro[H_{\tilde{q}}]{\tilde{q}(x,\cdot) + f}{\tilde{q}(y,\cdot) + f} + c = \tilde{q}(x,y) + f(x) + f(y) + \norm[\tilde{H}_q]{f}^2 + c
		\end{equation*}
		for all $x, y \in \R^+$. 
		
		If this holds for $r \ge 0$, then the cylindrical Wiener process $W$ with kernel $q$ is an $\cH^r$-valued $Q$-Wiener process and fulfills
		\begin{align*}
			&\E[\inpro[\cH^0]{W(1)}{u}\inpro[\cH^0]{W(1)}{v}] \\
			&\quad= \int_{\R^+ \times \R^+} \tilde q(x,y) \tilde u(x) \tilde v(y) \dd x \dd y + \int_{\R^+} a_u \tilde v(x) f(x) + a_v \tilde u (x) f(x) \dd x + a_u a_v (\norm[H_{\tilde q}]{f}^2 + c)
		\end{align*} 
		for all $u = \tilde u + a_u, v = \tilde v + a_v \in \cH^r = H^r \oplus \R$.
		
		If this holds for $r > 1/2$, then $x \mapsto W(t,x)$ is a mean square continuous and bounded random field for all $t \in [0,T]$ and $\E[W(1,x)W(1,y)] = q(x,y)$
		for all $x, y \in \R^+$.	
		
	\end{theorem}
	
	Note that if $H_q \hookrightarrow \cH^r$ for some $r > 1/2$, without the $\cL_2$ property of the embedding, there might not be a $Q$-Wiener process with covariance function $q$. 
	However, as long as  $H_q \xhookrightarrow[]{\cL_2} \cH^0 = L^2(\R^+) \oplus \R$, this theorem shows that we can still interpret $q$ as the covariance function of $W$ in a weaker $L^2$ sense. One might also consider embeddings in negatively weighted spaces. We do not pursue this direction but focus instead on constructing kernels $q$ such that $H_q \xhookrightarrow[]{\cL_2} H^r$.
	
	\begin{proof}[Proof of Theorem~\ref{thm:wiener-covariance-functions}]
		First, suppose that there is a separable RKHS $H_q$ such that $H_q \xhookrightarrow[]{\cL_2} \cH^r = H^r \oplus \R$. Let $(e_j)_{j=1}^\infty$ be an ONB of $H_q$ and write $e_j = f_j + c_j$ with $f_j \in \cH^r, c_j \in \R$. Since 
		\begin{equation}
			\label{eq:thm:wiener-covariance-functions:pf:2}
			\sum_{j = 1}^\infty \norm[\cH^r]{f_j}^2 + c_j^2 = \norm[\cL_2(H_q,\cH^r)]{I_{H_q \hookrightarrow \cH^r}} < \infty
		\end{equation}
		we have $(c_j)_{j=1}^\infty \in \ell^2$. Moreover, due to the fact that $f_j = e_j - c_j$, the evaluation operator $\delta_x$ is well-defined on $f_j$. Since
		\begin{equation*}
			\sum_{j=1}^\infty |f_j(x)|^2 = \sum_{j=1}^\infty |e_j(x) - c_j|^2 \le 2 \sum_{j=1}^\infty c_j^2 + 2 \sum_{j=1}^\infty |e_j(x)|^2 = 2 \norm[\ell^2]{(c_j)_{j=1}^\infty}^2 + 2 q(x,x) < \infty,
		\end{equation*}
		we may define a kernel $\tilde q$ by 
		\begin{equation*}
			\tilde q(x,y) = \sum_{j = 1}^\infty f_j(x) f_j(y)
		\end{equation*}
		for $x,y \in \R^+$. The kernel $\tilde q$ is symmetric and positive semidefinite. We define a Hilbert space $H$ by 
		\begin{equation*}
			H := \left\{ v = \sum_{j =1}^\infty v_j f_j,  (v_j)_{j =1}^\infty \subset \R \text{ such that } \norm{v}^2 := \sum_{j =1}^\infty v_j^2 < \infty \right\}.
		\end{equation*}
		An ONB of this space is $(f_j)_{j =1}^\infty$. Moreover, $H$ is a RKHS: for $v \in H$ and $x \in \R^+$,
		\begin{equation*}
			|v(x)|^2 \le \Big(\sum_{j =1}^\infty |v_j| |f_j(x)|\Big)^2 \le \sum_{j =1}^\infty |v_j|^2 \sum_{j =1}^\infty|f_j(x)|^2 = \norm[H]{v}^2 \tilde{q}(x,x).
		\end{equation*}
		Since also $\tilde q(x, \cdot) \in H$ and $\inpro[H]{\tilde q(x, \cdot)}{v} = v(x)$, we find that $H = H_{\tilde q}$ is the RKHS of $\tilde q$. The fact that $H_{\tilde q} \xhookrightarrow[]{\cL_2} H^r$ follows directly from~\eqref{eq:thm:wiener-covariance-functions:pf:2}. Let $\cJ = \{j \in \N \text{ such that } f_j \neq 0\}$. Then 
		\begin{align*}
			q(x,y) &= \sum_{j = 1}^\infty (f_j(x) + c_j)(f_j(y) + c_j) \\
			&= \sum_{j \in \cJ} f_j(x) f_j(y) + \sum_{j \in \cJ} c_j f_j(x) + \sum_{j \in \cJ} c_j f_j(y) + \sum_{j \in \cJ} c_j^2 + \sum_{j \in \cJ^c} c_j^2,
		\end{align*}
		where the split is justified by $(c_j)_{j = 1}^\infty, (f_j(x))_{j = 1}^\infty$ and $(f_j(y))_{j = 1}^\infty \in \ell^2$ for all $x,y \in \R^+$. By setting $f = \sum_{j \in \cJ} c_j f_j(x)$ and $c = \sum_{j \in \cJ^c} c_j^2$ we obtain one direction of the first claim of Theorem~\ref{thm:wiener-covariance-functions}. The other direction is obtained by analogous arguments.
		
		For the second claim, consider the same setting as above. Since $H_q \xhookrightarrow[]{\cL_2} \cH^r$, 
		\begin{equation}
			\label{eq:thm:wiener-covariance-functions:pf:1}
			\begin{split}
				\E[\inpro[\cH^r]{W(1)}{u}\inpro[\cH^r]{W(1)}{v}] = \inpro[\cH^r]{Qu}{v} &= \inpro[H_q]{I^*_{H_q \hookrightarrow \cH^r} u}{I^*_{H_q \hookrightarrow \cH^r} v} \\
				&= \sum^\infty_{j=1} \inpro[H_q]{I^*_{H_q \hookrightarrow \cH^r} u}{e_j} \inpro[H_q]{I^*_{H_q \hookrightarrow \cH^r} v}{e_j} \\
				&= \sum^\infty_{j=1} \inpro[\cH^r]{u}{e_j} \inpro[\cH^r]{v}{e_j}.
			\end{split}
		\end{equation}
		Applying this identity with $r = 0$ yields
		\begin{align*}
			&\E[\inpro[\cH^0]{W(1)}{u}\inpro[\cH^0]{W(1)}{v}] \\
			&\quad= \sum^\infty_{j=1} (\inpro[L^2]{\tilde u}{f_j} + a_u c_j) (\inpro[L^2]{\tilde v}{f_j} + a_v c_j) \\
			&\quad= \sum_{j \in \cJ} \inpro[L^2]{\tilde u}{f_j} \inpro[L^2]{\tilde v}{f_j} + \sum_{j \in \cJ} a_v c_j \inpro[L^2]{\tilde u}{f_j} + \sum_{j \in \cJ} a_u c_j \inpro[L^2]{\tilde v}{f_j} + \sum^\infty_{j=1} c_j^2 \\
			&\quad= \int_{\R^+ \times \R^+} \tilde q(x,y) \tilde u(x) \tilde v(y) \dd x \dd y + \int_{\R^+} \tilde a_u v(x) f(x) + a_v \tilde u (x) f(x) \dd x + a_u a_v (\norm[H_{\tilde q}]{f}^2 + c).
		\end{align*}
		The exchange of summation and integration in the last step is justified in the first case by 
		\begin{align*}
			&\sum_{j \in \cJ} \int_{\R^+ \times \R^+} |f_j(x)| |f_j(y)| |\tilde u(x)| |\tilde v(y)| \dd x \dd y \\
			&\quad\le \int_{\R^+} |\tilde u(x)|  \sqrt{\Big(\sum_{j \in \cJ} |f_j(x)|^2\Big)} \dd x
			\int_{\R^+} |\tilde v(y)|  \sqrt{\Big(\sum_{j \in \cJ} |f_j(y)|^2\Big)} \dd y \\
			&\quad\le \norm[\cL_2(H_q,L_2)]{I_{H_q \hookrightarrow L_2}} \norm[L^2]{\tilde v} \norm[L^2]{\tilde u} 
		\end{align*}
		using H\"older's inequality twice. The justification in the second and third cases is similar.
		
		For the third claim, we first note that $x \mapsto W(t,x)$ is a well-defined random field for all $t \in [0,T]$ due to \cite[Theorem~7.5.1]{HE15} and the fact that $(x,y) \mapsto m_r(y-x) + 1$ is continuous and bounded on $\R^+ \times \R^+$. The claim follows by taking $u = m_r(x-\cdot) + 1$ and $v = m_r(y-\cdot) + 1$ in~\eqref{eq:thm:wiener-covariance-functions:pf:1}. \qedhere
		
	\end{proof}
	
	Up to this point, we have let $W$ denote a general cylindrical Wiener process. We now introduce the key assumptions on the kernels of $B$ and $W$ in~\eqref{eq:spde-system} which we use to deduce regularity results for this system in Section~\ref{sec:regularity}. Due to the presence of the term $\Gamma^Z$ in~\eqref{eq:spde-system}, the assumptions mostly concern the covariance of $W$.
	
	\begin{assumption}
		\label{ass:noise-kernel}
		Let $B$ and $W$ in~\eqref{eq:spde-system} be cylindrical Wiener processes with kernels $q_B$ and $q_W$, respectively. Let $r_B > 1/2$ and $r_W \ge 0$. Assume the following for the RKHSs $H_{q_B}$ and $H_{q_W}$ on $\R^+$:
		\begin{enumerate}[label=(\roman*)]
			\item \label{ass:noise-kernel-i} The embedding $H_{q_B} \hookrightarrow \cH^{r_B}$ holds true.
			\item \label{ass:noise-kernel-iii} The embedding $H_{q_W} \xhookrightarrow[]{\cL_2} H^{r_W}(\R^+)$ holds true.
			\item \label{ass:noise-kernel-iv} There is an ONB $(e_j)_{j =1}^\infty$ of $H_{q_W}$ such that $\sum_{j = 1}^{\infty} \norm[L^\infty(\R^+)]{e_j}^2 < \infty$.
		\end{enumerate}
	\end{assumption}
	
	We also need the following assumption for our numerical analysis of approximations to $Y$ in Section~\ref{sec:numerical}.
	
	\begin{assumption}
		\label{ass:noise-kernel-w}
		Under the same conditions as in Assumption~\ref{ass:noise-kernel}, let $q_w$ be defined on $\R \times \R$. For some $s_{W} > 1/2$ and $\alpha > 0$, $H_{q_w}(\R) \hookrightarrow H^{s_{W},\alpha}(\R)$.
	\end{assumption}
	
	We end this section with an example of a class of kernels $q$ that fulfill all parts of these assumptions.  
	They are based on stationary kernels $q_s$ that are positive definite, integrable and continuous on $\R$. We recall that under these conditions, $q_s$ has a positive Fourier transform  $\hat{q}_s \colon \R \to \R^+$ which is also integrable on $\R$, see~\cite[Chapter~6]{W04}. 
	
	\begin{theorem}[Weight-stationary kernels]
		\label{thm:noise-kernel}
		Let $q_s \in \cL^1(\R) \cap \cC(\R)$ be a stationary positive definite kernel with a spectral density $\hat{q_s}$ such that for some constant $C < \infty$, $\sigma > 1/2$, $\hat{q_s}(\xi) \le C w_{-\sigma}(\xi)$ for all $\xi \in \R$. Let $w \colon \R \to \R^+$ be a continuous symmetric function such that the mapping $f(x) \mapsto {w_\alpha}(x) w(x) f(x)$, $x \in \R$, is bounded with respect to $\cL(H^\sigma(\R))$ for some $\alpha > 1/4$. Let $C_0$ be an arbitrary non-negative constant and let the kernel $q$ on $\R$ be defined by 
		\begin{equation*}
			q(x,y) := w(x) q_s(x-y) w(y), \, x, y \in \R.
		\end{equation*}
		Then $H_q = H_q(\R^+)$ fulfills 
		\begin{enumerate}[label=(\roman*)]
			\item \label{prop:noise-kernel-sobolev} $H_q \hookrightarrow H^{r,\alpha}(\R^+)$ for all $r \le \sigma$,
			\item \label{prop:noise-kernel-hs} $H_q \xhookrightarrow[]{\cL_2} H^{r}(\R^+)$ for all $r < \sigma - 1/2$, and
			\item \label{prop:noise-kernel-maxnorm} there is an ONB $(e_j)_{j =1}^\infty$ of $H_{q}$ such that $\sum_{j = 1}^{\infty} \norm[L^\infty(\R^+)]{e_j}^2 < \infty$. 
		\end{enumerate}
	\end{theorem}
	
	For the proof of this theorem, we need two lemmas. The first deals with the restriction of kernels on $\R$ to kernels on $\R^+$.  
	
	\begin{lemma}
		\label{lem:restrictions}
		Let $\cD \subseteq \R$ be a possibly unbounded interval and let $q$ be a kernel on $\R$ such that $H_q(\R)$ is separable. 
		Then $H_{q}(\cD)$ is separable and 
		\begin{enumerate}[label=(\roman*)]
			\item \label{lem:restrictions-i} if $H_q(\R) \xhookrightarrow[]{\cL_2} H^r(\R)$, then $H_{q}(\cD)  \xhookrightarrow[]{\cL_2} H^r(\cD)$ and
			\item \label{lem:restrictions-ii} if there is an ONB $(e^{\R}_j)_{j=1}^\infty$ of $H_q(\R)$ such that $\sum_{j = 1}^\infty \norm[L^\infty(\R)]{e^{\R}_j}^2 < \infty$, then there is an ONB $(e^\cD_j)_{j=1}^\infty$ of $H_q(\cD)$ such that $\sum_{j = 1}^\infty \norm[L^\infty(\cD)]{e^\cD_j}^2 < \infty$.
		\end{enumerate}
	\end{lemma}
	\begin{proof}
		We recall that $R_{\R \to \cD}(H_q(\R)) = H_q(\cD)$ with equal norms. Therefore, if $(e_j)_{j = 1}^\infty$ is an ONB of $H_q(\R)$, an ONB $(\tilde e_j)_{j = 1}^\infty$ of $H_{q}(\cD)$ is obtained by letting $\tilde e^\cD_j = R_{\R \to \cD} e_j$ for all $j \in \N$. This shows that $H_{q}(\cD)$ is separable.
		
		For~\ref{lem:restrictions-i}, we note that $I_{H_{q}(\cD) \hookrightarrow R_{\R \to \cD}(H^r(\R))} = R_{\R \to \cD} I_{H_q(\R) \hookrightarrow H^r(\R)} E_{\cD\to \R}$, where we write $E_{\cD\to \R} \colon H_{q}(\cD) \to H_{q}(\R) $ for the extension operator of functions on $\cD$ to functions on $\R$. This is Hilbert--Schmidt by the ideal property of this class. Since $R_{\R \to \cD}(H^r(\R)) = H^r(\cD)$, this finishes the proof of~\ref{lem:restrictions-i}. The statement in~\ref{lem:restrictions-ii} follows directly from the definition of the $L^\infty$-norm.
	\end{proof}
	
	The next lemma shows that part~\ref{prop:noise-kernel-sobolev} of Theorem~\ref{thm:noise-kernel} is satisfied and also that Assumption~\ref{ass:noise-kernel-w} is fulfilled.
	
	\begin{lemma}
		\label{lem:noise-kernel}
		Under the conditions of Theorem~\ref{thm:noise-kernel}, $H_{q}(\R) \hookrightarrow H^{\sigma,\alpha}(\R)$.
	\end{lemma}
	
	\begin{proof}
		The proof is divided into two steps. First, we show that $H_{m_{w,\sigma}}(\R) \hookrightarrow H^{\sigma,\alpha}(\R)$. Then we show that $H_{q}(\R) \hookrightarrow H_{m_{w,\sigma}}(\R)$. Here we write $m_{v,\sigma}$ for the weighted Mat\'ern kernel $(x,y) \mapsto v(x)m_\sigma(x-y)v(y)$, where $v$ is a given function on $\R$. 
		
		For the first step, it suffices to check that $m_{w_{-\alpha},\sigma} \gtrsim m_{w,\sigma}$. This means that there is some constant $C > 0$ such that for all $n \in \N$ and $d,x \in \R^n$, 
		\begin{equation*}
			\sum_{i,j=1}^{n} d_i \left(C w_{-\alpha}(x_i) m_{\sigma}(x_i-x_j) w_{-\alpha}(x_j) - w(x_i) m_{\sigma}(x_i-x_j) w(x_j) \right) d_j \ge 0.
		\end{equation*}
		Since $w_{-\alpha}$ is a strictly positive function, this is equivalent to showing that there is some constant $C > 0$ such that for all $n \in \N$ and $d,x \in \R^n$, 
		\begin{equation*}
			\sum_{i,j=1}^{n} d_i \big(C m_{\sigma}(x_i-x_j) - w_{\alpha}(x_i) w(x_i) m_{\sigma}(x_i-x_j) w(x_j) w_{\alpha}(x_j) \big) d_j \ge 0.
		\end{equation*}
		This in turn is equivalent to $H_{m_{w_{\alpha} w,\sigma}} \hookrightarrow H^\sigma$. 
		We now use the fact (see~\cite[Proposition~5.20]{PR16}) that $H_{m_{w_{\alpha} w,\sigma}} = \{u : u = w_{\alpha} w v \text{ for some } v \in H^\sigma\} = \{u : u = w_{\alpha} w v \text{ for some } v \in (H^\sigma_0)^\perp\}$. 
		Here $H^\sigma_0 : = \{u \in H^\sigma : w_{\alpha} w u = 0 \}$, which is closed in $H^\sigma$ by continuity of the evaluation operator. 
		The norm of this RKHS may be represented by $\norm[H_{m_{w_{\alpha} w,\sigma}}]{u} = \norm[H^\sigma]{v}$ for the unique $v \in (H^\sigma_0)^\perp$ such that $u = w_{\alpha} w v$. 
		Moreover, since we have assumed that the mapping $f \mapsto {w_\alpha} w f$ belongs to $\cL(H^\sigma)$, we have for all $u = w_{\alpha} w v$ that $\norm[H^\sigma]{u} \lesssim \norm[H^\sigma]{v}$. 
		This shows that $H_{m_{w_{\alpha} w,\sigma}} \hookrightarrow H^\sigma$.
		
		For the second step, let us note that the separable space $H_{q_s}$ is continuously embedded into $H^\sigma = H_{m_\sigma}$ (a straightforward consequence of~\cite[Theorem~10.12]{W04}, see also \cite[Corollary~10.13]{W04}). Therefore $m_\sigma \gtrsim q_s$. This directly implies $m_{w,\sigma} \gtrsim q$, which finishes the proof.
	\end{proof}
	
	We are now ready to prove the rest of Theorem~\ref{thm:noise-kernel}.  
	
	\begin{proof}[Proof of Theorem~\ref{thm:noise-kernel}.]
		
		Since $q$ is continuous, $H_q(\R)$ is separable. With this in mind, let us first show~Theorem~\ref{thm:noise-kernel}\ref{prop:noise-kernel-hs}. By Lemmas~\ref{lem:restrictions} and~\ref{lem:noise-kernel}, it suffices to show that~$H^{\sigma, \alpha}(\R^+) \xhookrightarrow[]{\cL_2} H^{r}(\R^+)$. We do this by combining the approximation number properties~\eqref{eq:approximation-bound-by-entropy} and~\eqref{eq:hs-norm-approximation-numbers} with the observation that $(\epsilon_n(I_{H^{\sigma,\alpha}(\R) \hookrightarrow H^r(\R)}))_{n=1}^\infty \in \ell^2$.
		This observation follows from the results of~\cite{HT94}, wherein entropy numbers of embeddings between weighted Besov-type spaces $B^s_{p,q}(\R)$ are studied. We note that $H^r(\R) = B^r_{2,2}(\R)$ with equivalent norms for all $r \ge 0$ \cite[Theorems~2.3.9 and~2.5.6]{T83}. With this in mind, we have by \cite[Theorem~4.2]{HT94} that $\epsilon_n(I_{H^{\sigma,\alpha}(\R) \hookrightarrow H^r(\R)})$ is bounded from above and below by a constant (independent of $n$) times $n^{-(\sigma-r)}$ in the case that $\sigma < r + 2\alpha$ and $n^{-2\alpha}$ in the case that $\sigma > r + 2\alpha$. In either of these two cases, therefore, the number is bounded by $n^{-1/2-\epsilon}$ for some $\epsilon > 0$ so that $H^{\sigma, \alpha} \xhookrightarrow[]{\cL_2} H^{r}$. In case $2\alpha = \sigma-r > 1/2$ we can find some $\tilde\alpha \in (1/4,\alpha)$ so that $H^{\sigma, \alpha} \hookrightarrow H^{\sigma, \tilde \alpha} \xhookrightarrow[]{\cL_2} H^{r}$, hence we arrive at the same conclusion. We note at this point that the results of~\cite{HT94} are for spaces of complex-valued functions, but it follows directly from the definition that the entropy numbers of embeddings of our real-valued spaces are bounded by those of the complex-valued spaces.
		
		For the claim~\ref{prop:noise-kernel-maxnorm}, we recall that an operator $T \in \cL(H,E)$, where $H$ is a separable Hilbert space and $E$ is a Banach space, is said to be $\gamma$-radonifying if 
		\begin{equation*}
			\norm[\gamma(H,E)]{T}^2 := \E\left[\Bignorm[L^\infty]{\sum_{j = 1}^\infty \xi_j T e_j}^2\right] < \infty
		\end{equation*}
		where this operator norm is independent of the choice of ONB $(e_j)^\infty_{j=1}$ of $H$ and sequence of iid standard Gaussian random variables $(\xi_j)_{j=1}^\infty$ \cite[Corollary~3.21]{vN10}. We will show that $I_{H_{q}(\R) \hookrightarrow L^\infty(\R)} \in \gamma(H_{q}(\R),L^\infty(\R))$, from which the result follows by \cite[Theorem~3.23]{vN10} and Lemma~\ref{lem:restrictions}. First, we note that the class of $\gamma$-radonifying operators is an operator ideal \cite[Theorem~6.2]{vN10}. Moreover, $B^0_{\infty,1}(\R) \hookrightarrow L^\infty(\R)$, a consequence of~\cite[Proposition~2.5.7]{T83}. We also have 
		\begin{equation*}
			I_{H_{q}(\R) \hookrightarrow L^\infty(\R)} =  I_{B^0_{\infty,1}(\R) \hookrightarrow L^\infty(\R)} I_{H^{\sigma,  \alpha}(\R) \hookrightarrow B^0_{\infty,1}(\R)} I_{H_{q}(\R) \hookrightarrow H^{\sigma, \alpha}(\R)}.
		\end{equation*}
		Therefore, by Lemma~\ref{lem:noise-kernel}, it suffices to note that $I_{H^{\sigma, \alpha}(\R) \hookrightarrow B^0_{\infty,1}(\R)} \in \gamma(H^{\sigma, \alpha}(\R),B^0_{\infty,1}(\R))$. This follows from~\cite[(2.2)]{KL19} if it holds that $(n^{-1/2} \epsilon_n(I_{H^{\sigma, \alpha}(\R)) \hookrightarrow L^\infty(\R)}))_{n=1}^\infty \in \ell^1$. Using~\cite[Theorem~4.2]{HT94}, \eqref{eq:operator-norm-property} and~\eqref{eq:multiplicative-property}, we see that $\epsilon_n(I_{H^{\sigma, \alpha}(\R) \hookrightarrow L^\infty(\R)})$ is bounded by a constant times $n^{-\sigma}$ if $\sigma - 1/2 < 2\alpha$, %
		by $n^{-2\alpha-1/2} (1+\log(n))$ if $\sigma - 1/2 > 2\alpha$ %
		and by $(n/(1+\log(n)))^{-\sigma}$ if $2\alpha = \sigma - 1/2$. %
		From this the result follows. %
	\end{proof}
	
	\begin{remark}
		\label{rem:matern}
		For the weight in Theorem~\ref{thm:noise-kernel}, possible choices include $w = w_{-\alpha}$ with $\alpha > 1/4$ (or a rescaling thereof) or smooth bump functions. For the stationary kernel $q_s$, one might choose it to be a Mat\'ern kernel. This class includes exponential kernels as special cases and is defined by
		\begin{equation*}
			q_s(x) := \zeta 2^{1-\nu}/\Gamma(\nu) ({\sqrt{2\nu}|x|}/{\mu})^\nu K_\nu((\sqrt{2\nu}|x|)/\mu)
		\end{equation*}
		for $x \in \R$. Here $K_\nu$ is the modified Bessel function of the second kind with order $\nu > 0$ and $\Gamma$ is the Gamma function. It is scaled by the positive parameters, the variance $\zeta$ and the correlation length $\mu$. The assumption on $q_s$ in Theorem~\ref{thm:noise-kernel} holds with $\sigma = \nu + 1/2$, see, e.g., \cite[Example~6.8]{LPS14}. The exponential kernel is obtained with $\nu = 1/2$.
	\end{remark}
	\begin{remark}
		Since we are mainly interested in Wiener processes on the half-space $\R^+$, we restricted ourselves to this case. Theorem~\ref{thm:noise-kernel} can be seen to hold true also in $\R^d$, with the bounds on $\sigma$ and $\alpha$ replaced by dimension-dependent constants. Moreover, if the application at hand calls for genuinely non-stationary kernels, H\"older conditions could be analyzed using similar techniques.
	\end{remark}
	
	\section{Existence, uniqueness and Sobolev regularity of the HEIDIH model}
	
	\label{sec:regularity}
	
	In this section we discuss existence and uniqueness of the two components $X$ and $Y$ in the HEIDIH model~\eqref{eq:spde-system}. We start with the volatility process $Y$, which is the mild solution 
	\begin{equation}
		\label{eq:mild-stochastic-heat}
		Y(t) = \cU(t) Y(0) + \int^t_0 \cU(t-s) \dd W(s), \, t \in [0,T],
	\end{equation}
	to the stochastic heat equation
	\begin{equation}
		\label{eq:stochastic-heat}
		\dd Y(t) = \cA Y(t) \dd t + \dd W(t), \, t \in [0,T],
	\end{equation}
	starting at $Y(0)$ on $\R^+$. We recall that $\cA$ denotes the Laplacian $a\Delta$ with homogeneous zero boundary conditions at $x=0$ of either Dirichlet or Neumann type scaled by a diffusion coefficient $a$. 
	We derive spatial regularity results in a fractional Sobolev sense along with a pointwise-in-space Hölder regularity result. The latter is needed for the numerical approximation in the next section. With these results in place, we consider the full system~\eqref{eq:spde-system} and derive existence, uniqueness and fractional Sobolev regularity for the mild solution 
	\begin{equation}
		\label{eq:mild-transport}
		X(t) = \cS(t) X(0) + \int^t_0 \cS(t-s) \Gamma^Z(s) \dd B(s)
	\end{equation}
	under suitable assumptions on $X(0)$, $Z$ and $B$.
	
	\subsection{The stochastic heat equation}
	
	We pose~\eqref{eq:stochastic-heat} as an equation in the space $L^2(\R^+)$, but impose sufficient regularity of the components to guarantee that it takes values in $H^r(\R^+) \subset \cH^r$ for $r > 1/2$. We let the operator $\cA = a\Delta$ be the realization of the Laplacian in $L^2(\R^+)$. Since the spectrum $\sigma(-\cA)$ is contained in $[0,\infty]$, $-\cA + \epsilon$ is a sectorial operator for all $\epsilon > 0$. Fractional powers of the self-adjoint and positive definite operator $-\cA + \epsilon$ are therefore well-defined and $-\cA$ is the generator of an analytic semigroup \cite[Section~2.7,Theorem~3.1]{Y10}. By the proofs of~\cite[Theorems~1-2]{F67}, we have the following characterizations of the domains of the operators $(-\cA + \epsilon)^{r/2}$ for small $r \ge 0$:
	\begin{equation}
		\label{eq:sobolev_id}
		\dot{H}^r := \dom((-\cA + \epsilon)^{r/2}) = H^r(\R^+) \text{ if } r \in [0,\theta).
	\end{equation}
	In the case of Dirichlet boundary conditions we have $\theta = 1/2$, while in the case of Neumann boundary conditions we have $\theta = 3/2$. Here, we also made implicit use of the fact that the Bessel potential norm and the so called Sobolev--Slobodeckij norm (see, e.g., \cite[(1.70)]{Y10}) are equivalent in our setting, which follows from the existence of an extension operator, see, e.g., \cite[Theorem~1.33]{Y10}. The equality~\eqref{eq:sobolev_id} should be understood as the existence of a canonical isomorphism with norm equivalence, with respect to the Hilbert space structure of~$\dot{H}^r \hookrightarrow L^2(\R^+)$ when this is equipped with the graph norm $\norm[\dot{H}^r]{\cdot} :=	\norm[L^2]{(-\cA+\epsilon)^{r/2}\cdot}$. The sequence $(\dot{H}^r)_{r \ge 0}$ consists of spaces that are continuously and densely embedded into one another. By interpolation techniques (cf.\ \cite[Theorem~16.3]{Y10}) one may show that $\dot{H}^r \hookrightarrow H^r$ for all $r \ge 0$. 
	
	Let the Wiener process $W$ in~\eqref{eq:spde-system} have covariance kernel $q_w$ %
	fulfilling Assumption~\ref{ass:noise-kernel}\ref{ass:noise-kernel-iii} for some $r_w \ge 0$. If we regard $W$ as a $Q$-Wiener process in $L^2(\R^+)$, its covariance operator is given by $Q = I_{H_{q_w}\hookrightarrow L^2} I^*_{H_{q_w}\hookrightarrow L^2}$. Since $Q^{1/2}(L^2(\R^+)) = H_{q_w}(\R^+)$, we have $Q^{1/2}(L^2(\R^+)) \xhookrightarrow[]{\cL_2} \dot{H}^r$ if and only if $H_{q_w}(\R^+) \xhookrightarrow[]{\cL_2} H^r$, provided that $r < \min(r_W,\theta)$. Therefore, \cite[Assumption~3]{JR12} is fulfilled and we obtain the following existence, uniqueness and regularity result.
	\begin{theorem}[{\cite[Theorem~1]{JR12}}]
		\label{thm:heat-regularity}
		Suppose that the covariance kernel $q_w$ of $W$ fulfills Assumption~\ref{ass:noise-kernel}\ref{ass:noise-kernel-iii} and let $\rho := \min(r_W,\theta)$. Suppose also that $Y(0) \in \dot{H}^{\rho + 1}$. Then there exists a unique mild solution $Y = (Y(t))_{t \in [0,T]}$ to~\eqref{eq:mild-stochastic-heat}. Moreover, for all $r < \rho+1 $, there is a continuous modification of $Y$ in $\dot{H}^r$ such that $\sup_{t \in [0,T]} \E[\norm[\dot{H}^r]{Y(t)}^2] < \infty$ and there is a constant $C < \infty$ such that for all $t_1,t_2 \in [0,T]$ 
		\begin{equation}
			\label{eq:holder-in-dot-h}
			\E[\norm[\dot{H}^r]{Y(t_1) -  Y(t_2)}^2]^{1/2} \le C |t_2 - t_1|^{\min(1, \rho + 1 - r)/2}.
		\end{equation}
	\end{theorem}
	\begin{remark}
		In~\cite[Theorem~1]{JR12}, the regularity is restricted to $r < 2$. This is only due to the nonlinearities considered in \cite{JR12}, which are not present in our setting.
	\end{remark}

	By the Sobolev embedding theorem, we obtain from the estimate~\eqref{eq:holder-in-dot-h} that there is a constant $C < \infty$ such that for all $x \in \R^+$ and $t_1,t_2 \in [0,T]$, $\E[|Y(t_1,x)-Y(t_2,x)|^2] \le C |t_1-t_2|^{\min(1,r)}$ for all $r < \rho + 1/2$. If $r_W=0$, we therefore obtain that $Y(\cdot,x)$ is H\"older continuous with respect to the root mean squared norm with exponent up to $1/4$. Similarly, $Y(t,\cdot)$ is H\"older continuous with exponent up to $1/2$. However, also for such relatively rough noise, the H\"older exponents can be improved, provided that Assumption~\ref{ass:noise-kernel}\ref{ass:noise-kernel-iv} is fulfilled. This we show in the next proposition. A concrete example of a covariance kernel in this setting includes a weighted Mat\'ern kernel with smoothness parameter $\nu$ close to $0$, see Remark~\ref{rem:matern}. For the proof, we recall that the semigroup associated with $-\cA$ has an explicit representation by the reflection method. Indeed, with $\Phi(x,t) = e^{-x^2/(4 a t)}/\sqrt{4 \pi a t}$, $x \in \R, t > 0$, denoting the heat kernel,
	\begin{equation}
		\label{eq:halfline-heat-kernel}
		(\cU(t)v)(x) = \int_0^\infty (\Phi(x-y,t) \pm \Phi(x+y,t)) v(y) \dd y, \, t > 0, x \ge 0,
	\end{equation}
	for $v \in L^2(\R^+)$. The sign is positive for Neumann boundary conditions and negative for Dirichlet boundary conditions. This can be seen by the density of smooth functions with compact support in $L^2(\R^+)$. Alternatively, we may write
	\begin{equation}
		\label{eq:reflection-formula}
		(\cU(t)v)(x) = \int_{-\infty}^\infty \Phi(x-y,t) \tilde v(y) \dd y, \, t > 0, x \ge 0,
	\end{equation}
	where $\tilde v$ is an even extension of $v$ around $0$ for Neumann boundary conditions and an odd extension in the Dirichlet case.
	
	\begin{proposition}
		\label{prop:holder}
		Under the same conditions as in Theorem~\ref{thm:heat-regularity} together with Assumption~\ref{ass:noise-kernel}\ref{ass:noise-kernel-iv}, there exists for all $r < 1$ a constant $C < \infty$ such that 
		\begin{enumerate}[label=(\roman*)]
			\item \label{eq:time-holder} $\E[|Y(t,x)-Y(s,x)|^2]^{1/2} \le C (1 + s^{-r/2}\norm[L^\infty]{Y(0)}) |t-s|^{r/2}$, and
			\item \label{eq:space-holder} $\E[|Y(s,x)-Y(s,y)|^2]^{1/2} \le C (1 + s^{-r/2} \norm[L^\infty]{Y(0)}) |x-y|^{r}$
		\end{enumerate}
		for all $x,y \in \R^+$ and $T \ge t \ge s > 0$.
	\end{proposition}
	
	\begin{proof}	
		By differentiation of $\Phi(x,\cdot)$, we have, in the notation of~\eqref{eq:reflection-formula},
		\begin{equation*}
			|(\cU(t)v)(x)- (\cU(s)v)(x)| \le \int^\infty_{-\infty} \int^t_s \frac{(x-y)^2 + 2r}{4 r^2} \Phi(x-y,r) \tilde{v}(y) \dd r \dd y.
		\end{equation*}
		Since $x \mapsto \Phi(x,r)$ and $x \mapsto x^2\Phi(x,r)$ integrates to $1$ and $2 a r$, we obtain $\norm[L^\infty]{(\cU(t)-\cU(s))v} \lesssim (t-s) (s)^{-1} \norm[L^\infty]{v}$ for $0 < s \le t$. Moreover, $\cU$ is stable in $L^\infty$ in the sense that $\norm[L^\infty]{\cU(t)v} \le \norm[L^\infty]{v}$ for all $t \ge 0$. Combining these two estimates yields, for all $r \in [0,1]$, a constant $C < \infty$ such that for all $0<s\le t < \infty$ and $v \in L^\infty(\R^+)$,
		\begin{equation}
			\label{eq:prop:time-holder:1}
			\norm[L^\infty]{(\cU(t)-\cU(s))v} \le C s^{-r} (t-s)^{r} \norm[L^\infty]{v}.
		\end{equation}
		We now prove claim~\ref{eq:time-holder} for $Y(0) = 0$. The general case follows immediately by the estimates above, noting that $\norm[L^\infty]{Y(0)} < \infty$ by the Sobolev embedding theorem.
		
		Since $\delta_x$ is continuous on $H^r$ for $r > 1/2$, Theorem~\ref{thm:heat-regularity} justifies moving $\delta_x$ inside the It\^o integral below, so that by the It\^o isometry 
		\begin{equation}
			\label{eq:prop:time-holder:ito-calc}
			\begin{split}
				\E[|Y(t,x)-Y(s,x)|^2] &= 
				\E[|\delta_x (Y(t)-Y(s))|^2] \\
				&\lesssim  \E\left[\left| \int_{s}^{t} \delta_x \cU(t-r) \dd W(r) \right|^2\right] \\
				&\quad+  \E\left[\left| \int_0^{s} \delta_x (\cU(t-r) - \cU(s-r)) \dd W(r) \right|^2\right] \\
				&=\int_{s}^{t} \sum_{j=1}^\infty |(\cU(t-r) e_j)(x)|^2 \dd r \\ 
				&\quad+ \int_0^{s} \sum_{j=1}^\infty |((\cU(t-r) - \cU(s-r)) e_j)(x)|^2 \dd r.
			\end{split}
		\end{equation}
		Here, we let $(e_j)_{j=1}^\infty$ be an ONB of $H_{q_w}$ such that $\sum_{j = 1}^{\infty} \norm[L^\infty]{e_j}^2 < \infty$, see Assumption~\ref{ass:noise-kernel}\ref{ass:noise-kernel-iv}. Then, the claimed estimate follows by making use of~\eqref{eq:prop:time-holder:1} for the second term and stability of $\cU$ in $L^\infty$ for the first term.
		
		Claim~\ref{eq:space-holder} follows along the same lines.  By differentiation of $\Phi(t,\cdot)$, one may show that $|(\cU(s)v)(x)-(\cU(s)v)(y)| \le |x-y|\norm[L^\infty]{v}/\sqrt{4 \pi a s}$. By interpolation between this estimate and the stability estimate in $L^\infty$, along with an argument analogous to~\eqref{eq:prop:time-holder:ito-calc}, the claim is proven.
	\end{proof}
	
	\subsection{Regularity of the HEIDIH model}
	
	For the full model~\eqref{eq:spde-system}, we let $B$ and $W$ be Wiener processes with kernels $q_B$ and $q_W$. For $q_B$, we assume that Assumption~\ref{ass:noise-kernel}\ref{ass:noise-kernel-i} is fulfilled and for the components of $Y$ we consider the same setting as in~Theorem~\ref{thm:heat-regularity}. For the process $Z$, we make the following assumption.
	
	\begin{assumption}
		\label{ass:Z}
		Let, for some $r > 1/2$, $Z = (Z(t))_{t \in [0,T]}$ be an adapted stochastic process in $\cH^r$ fulfilling $\norm{Z(t)} = 1$ for all $t \in [0,T]$. Let $Z$ be such that the $\cL_2(\cH^r)$-valued process $\Gamma^Z$ given by $\Gamma^Z(t) := Y(t) \otimes Z(t), t \in [0,T]$, is predictable and fulfills
		\begin{equation*}
			\int^T_0 \E[\norm[\cL_2(H_{q_B},\cH^r)]{\Gamma^Z(t)}^2] \dd t < \infty.
		\end{equation*}
	\end{assumption}
	
	In~\cite{BS18}, two examples for $Z$ are studied in detail. First, the constant setting $Z(t) := \eta \in \cH^r$ for all $t \in [0,T]$ with $\norm[\cH^r]{\eta} = 1$. Second, the setting in which $\Gamma^Z(t)$ is the unique positive definite square root of $Y(t) \otimes Y(t)$, i.e., letting 
	\begin{equation*}
		Z(t) :=  \begin{cases}
			Y(t)/\norm[\cH]{Y(t)} &\text{ for } Y(t) \neq 0, \\
			1 &\text{ otherwise,} 
		\end{cases}
	\end{equation*}
	see~\cite[Proposition~4]{BS18}. Note that $1 \in \cH^r$ with $\norm[\cH^r]{1} = 1$. For these two examples, $\Gamma^Z$ is an adapted process with a continuous version. This follows from the continuity of $Y$, see~\cite[Lemma~5]{BS18}. We note that the authors of~\cite{BS18} assume that the covariance operator of $B$ is trace-class when confirming that the integral in Assumption~\ref{ass:Z} is finite for these two examples, see~\cite[Lemma~8]{BS18}. For us, this translates to $H_{q_B} \xhookrightarrow[]{\cL_2} \cH^r$. However, it is sufficient to require that $H_{q_B} \hookrightarrow \cH^r$, since
	\begin{equation}
		\label{eq:Gamma-Z-norm}
		\begin{split}
			\int^T_0 \norm[\cL_2(\cH^r)]{\Gamma^Z(t)}^2 \dd t &= \sum_{j = 1}^\infty \int^T_0 |\inpro[\cH^r]{Z(t)}{e_j}|^2 \norm[\cH^r]{Y(t)}^2 \dd t \\
			&= \sum_{j = 1}^\infty \int^T_0 |\inpro[H_{q_w}]{I^*_{H_{q_w} \hookrightarrow \cH^r} Z(t)}{e_j}|^2 \norm[\cH^r]{Y(t)}^2 \dd t \\
			&= \int^T_0 \norm[H_{q_w}]{I^*_{H_{q_w} \hookrightarrow \cH^r} Z(t)}^2 \norm[\cH^r]{Y(t)}^2 \dd t \\
			&\le \norm[\cL(H_{q_w},\cH^r)]{I_{H_{q_w} \hookrightarrow \cH^r}} \int^T_0 \norm[\cH^r]{Y(t)}^2 \dd t
		\end{split}
	\end{equation}
	almost surely and $\norm[\cL_2(H_{q_B},\cH^r)]{\Gamma} \lesssim \norm[\cL_2(\cH^r)]{\Gamma}$ for $\Gamma \in \cL_2(\cH^r)$.
	
	The regularity of $X$, as well as its existence and uniqueness as a process in $\cH^r$, now follows from \cite[Theorem~6.10]{DPZ14}. This makes use of the fact that $\cS$ is a contraction $C_0$-semigroup on $\cH^r$. This is true under the convention $(\cS(t)f)(x) = \tilde f(x+t) + a$ for $f = \tilde f + a \in H^r \oplus \R = \cH^r$, see~\cite[Appendix~A.1]{ET07}.
	
	We summarize these remarks in the following theorem.
	
	\begin{theorem}
		\label{thm:transport-regularity}
		Suppose that for some $r > 1/2$, Assumption~\ref{ass:Z} is fulfilled and $X(0) \in \cH^r$. Suppose also that the conditions of Theorem~\ref{thm:heat-regularity} as well as Assumption~\ref{ass:noise-kernel}\ref{ass:noise-kernel-i} hold true for some $\rho = r_B = r$. Then there exists a unique mild solution $X = (X(t))_{t \in [0,T]}$ given by~\eqref{eq:mild-transport} which is $\cH^r$-valued and has a continuous modification.
	\end{theorem}
	
	We end this section by commenting on the special case that $Z(t) := \eta \in \cH^r$ for all $t \in [0,T]$. Let $(e_j)_{j=1}^\infty$ be an ONB of $H_{q_B}$. By~\eqref{eq:wiener-construction} and approximation by simple functions, it follows from independence of $B$ and $W$ that for $x \in \R^+$ 
	\begin{equation}
		\label{eq:X-with-eta-simple-version}
		\begin{split}
			\delta_x \int^t_0 \cS(t-s) \Gamma^Z(s) \dd W(s) &= \sum^\infty_{j=1} \int^t_0 Y(s,x + t-s) \inpro{\eta}{e_k} \dd \beta_k(s) \\
			&= \norm[H_{q_B}]{I^*_{H_{q_B} \hookrightarrow \cH^r} \eta} \int^t_0 Y(s,x + t-s) \dd \tilde \beta(s),
		\end{split} 
	\end{equation}
	where 
	\begin{equation}
		\label{eq:beta-tilde}\tilde \beta(t) := \norm[H_{q_B}]{I^*_{H_{q_B} \hookrightarrow \cH^r} \eta}^{-1}  \sum_{j = 1}^\infty \inpro{\eta}{e_j} \beta_j(t), \, t \in [0,T],  
	\end{equation} 
	is a real-valued Wiener process. We see that in this case the stochastic volatility is directly given by $Y$ while $\eta$ is a global (in $x$) scaling parameter. A convenient choice for $\eta$ is to first set $\tilde \eta := \sum^N_{j=1} a_j (1 + m_r(x_j-\cdot))$ for some constants $a_j$ and points $x_j \in \R^+$, $j = 1, \ldots, N \in \N$, where we recall that $(x,y) \mapsto 1 + m_r(x - y)$ is the kernel of $\cH^r$. This class of functions is dense in $\cH^r$, see \cite[Corollary~2]{BT04}. Then $$\norm[\cH^r]{\tilde\eta}^2 = \sum^N_{i,j=1} a_i (1 + m_r(x_i-x_j))a_j$$ so with $\eta := \tilde \eta / \norm[\cH^r]{\tilde \eta}$ we obtain the expression
	\begin{equation*}
		\norm[H_{q_B}]{I^*_{H_{q_B} \hookrightarrow \cH^r} \eta}^2 = \frac{\sum^N_{i,j=1} a_i q_B(x_i,x_j) a_j}{\sum^N_{i,j=1} a_i (1 + m_r(x_i-x_j))a_j}
	\end{equation*}
	which can readily be computed.
	
	\begin{remark}
		\label{rem:example-figure}
		An example of the setting above with $Z(t) = \eta \in \cH^r$, is shown in Figures~\ref{subfig:Y} and~\ref{subfig:X}. Here $X(0)$ and $Y(0)$ are smooth bump functions while $q_W(x,y) = w(x) q_s(x-y) w(y)$, where $w(x) = 10^{-1/2} (1 + x^2)^{-3/4}$ and $q_s$ is a Mat\'ern kernel with $\nu = \mu = 0.1$ and unit variance, see Remark~\ref{rem:matern}. In $\cA = a\Delta$ we have set $a = 0.05$. 
		By Theorems~\ref{thm:heat-regularity} and~\ref{thm:transport-regularity}, we obtain that both $X$ and $Y$ take values in $\cH^{r}$ for $r = 1.1$. We have chosen $\eta$ and $q_B$ in such a way that $H_{q_B} \hookrightarrow \cH^r$ with $\norm[H_{q_B}]{I^*_{H_{q_B} \hookrightarrow \cH^r} \eta} = 1$.
	\end{remark}

	\section{Numerical approximation}
	\label{sec:numerical}
	
	A rigorous understanding of convergence rates in numerical approximations of SPDEs is vital for several reasons. First, they allow us to theoretically understand different approximation methods with respect to computational speed. Second, if the approximations are used in Monte Carlo methods, the convergence rates can help us to choose sample sizes in an optimal way \cite{L16}. Third, understanding how different sources of errors contribute in different ways to the overall error lets us tune the various parts of the algorithm, so that each error source contributes equally, leading to a computationally efficient algorithm. We return to this last part in Section~\ref{subsec:fem-section}. 
	
	Arguably, most literature on numerical approximation of SPDEs focus on the analysis of errors as measured in a spatial $L^2$-norm in some mean square sense, see, e.g., \cite{JK09,K14,LPS14} and the references within. This is especially true for finite element approximations. In the context of SPDE models of forward prices, the convergence discussed for the discontinuous Galerkin finite element method in~\cite{BB14} is of this type. In our setting, this would translate to errors of the form $\E[\norm[L^2(\R^+)]{X(t_i)-\hat X(t_i)}^2]$ in the numerical approximation of the HEIDIH model, where $\hat X$ refers to a given approximation and $(t_i)_{i = 1}^\infty$ are points on a temporal grid. However, the spatial argument $x$ in $X(t,x)$ refers to time to maturity. To price contracts delivering a commodity at a fixed time, therefore, we should instead consider pointwise-in-space errors, i.e., errors of the form $\E[|{X(t_i,x_j)-\tilde X(t_i,x_j)}|^2]^{1/2}$ where $(t_i,x_j)_{i,j=1}^\infty$ are points on a space-time grid. This is the approach we take below, where we first prove a general error decomposition formula for $X$. We then restrict ourselves to the setting that $Z(t) = \eta$ for all $t \in [0,T]$ and $Y$ is equipped with Dirichlet boundary conditions. In this case we derive sharp convergence rates with respect to pointwise-in-space errors for fully discrete approximations of $X$ and $Y$.
	
	\subsection{Error decomposition formula for a finite difference scheme in space and time}
	
	\label{subsec:error-decomp}
	
	We consider approximating the process $X$ on a spatiotemporal grid $(t_i,x_j)_{i,j \in \N_0} = (i k,jk)_{i,j \in \N_0}$ using finite differences with equal uniform step sizes $k$ in both time and space. The approximation is denoted by $\hat X$ and we set $\hat X(0) = X(0) \in \cH^r$. Otherwise, $\hat X$ is defined on the grid by the recursion
	\begin{equation}\\
		\label{eq:X-recursion}
		\hat X(t_{i+1},x_j) = \hat X(t_{i},x_{j+1}) + \delta_{x_j} (\hat Y(t_{i})\otimes\hat Z(t_{i})) \Delta B^{i}, \, i,j \ge 0,
	\end{equation}
	where the $\cH^r$-valued stochastic processes $\hat Y, \hat Z$ are given approximations $Y$ and $Z$, respectively, and $\Delta B^i = B(t_{i+1})-B(t_i)$. Iterating this, we obtain in closed form
	\begin{equation}
		\label{eq:X-closed-form-approximation}
		\hat X(t_{n},x_j) = \delta_{x_j} \cS(t_{n}) X(0) + \sum^{n-1}_{i=0} \int^{t_{i+1}}_{t_i} \delta_{x_j} \cS(t_{n} - t_{i+1}) (\hat Y(t_{i}) \otimes \hat Z(t_i) ) \dd B(r).
	\end{equation}
	This leads to the following error decomposition. We remark that since we do not use any particular property of $Y$, this result applies to any Heston stochastic volatility model of forward prices in the sense of~\cite{BDNS21,BS18}. Note also that, due to the use of a grid with equal step sizes in both space and time, the part of~\eqref{eq:mild-transport} associated with $X(0)$ is solved exactly. 
	\begin{proposition}
		\label{prop:X-error-decomp}
		Suppose that $\hat Y, \hat Z$ are such that $t \mapsto \hat{\Gamma}^{\hat{Z}_t} = \hat Y(t) \otimes \hat Z(t)$ is a predictable  $\cL_2(\cH^r)$-valued process fulfilling 
		\begin{equation*}
			\int^T_0 \E[\norm[\cL_2(H_{q_B},\cH^r)]{\hat{\Gamma}^{\hat{Z}}(t)}^2] \dd t< \infty.
		\end{equation*}
		The approximation error $X(t_{n},x_{j}) - \hat X(t_{n},x_{j})$ at the points $t_{n} =n k \le T$ and $x_j = jk \in \R^+$ can be decomposed as
		\begin{align*}
			\bigE{\left|X(t_{n},x_{j}) - \hat X(t_{n},x_{j})\right|^2} &= \sum^{n-1}_{i=0}  \int^{t_{i+1}}_{t_i} \E \Big[Y(r,t_{n} + x_{j} - r)^2 \norm[H_{{q_B}}]{I^*_{H_{q_B} \hookrightarrow \cH^r} \hat Z(r)}^2 \\ &\quad -2 Y(r,t_{n} + x_{j} - r) \hat Y(t_i,t_{n} + x_{j} -t_{i+1}) \\ &\qquad\times \inpro[H_{{q_B}}]{I^*_{H_{q_B} \hookrightarrow \cH^r} Z(r)}{I^*_{H_{q_B} \hookrightarrow \cH^r} \hat Z(t_i)} \\
			&\quad+\hat Y(t_i,t_{n} + x_{j} -t_{i+1})^2 \norm[H_{{q_B}}]{I^*_{H_{q_B} \hookrightarrow \cH^r} \hat Z(t_i)}^2\Big] \dd r.
		\end{align*}
		In the special case that $Z(t) = \hat Z (t) = \eta \in \cH^r$,
		\begin{align*}
			&\bigE{\left|X(t_{n},x_{j}) - \hat X(t_{n},x_{j})\right|^2} \\
			&\quad= \norm[H_{q_B}]{I^*_{H_{q_B} \hookrightarrow \cH^r} \eta}^2 \sum^{n-1}_{i=0} \int^{t_{i+1}}_{t_i} \bigE{\left|Y(r,t_{n} + x_{j}-r) - \hat Y(t_i,t_{n} + x_{j}- t_{i+1})\right|^2} \dd r.
		\end{align*}
	\end{proposition}
	
	\begin{proof}
		Note that since $\delta_{x_j} \in \cL(\cH^r,\R)$, the It\^o isometry immediately yields
		\begin{align*}
			&\bigE{\left|X(t_{n},x_{j}) - \hat X(t_{n},x_{j})\right|^2} \\
			&\hspace{0.25em}= \sum^{n-1}_{i=0}  \int^{t_{i+1}}_{t_i} \E \Big[ \norm[\cL_2(H_{q_B},\cH^r)]{ Y(r,t_{n} + x_{j} - r) \inpro[\cH^r]{Z(r)}{\cdot} - \hat Y(r,t_{n} + x_{j} - t_{i+1}) \inpro[\cH^r]{\hat{Z}(r)}{\cdot} }^2\Big] \dd r
		\end{align*}
		and the result follows as in~\eqref{eq:Gamma-Z-norm}.
	\end{proof}
	
	\subsection{A localized finite element discretization of the stochastic heat equation}
	\label{subsec:fem-section}
	
	The second part of Proposition~\ref{prop:X-error-decomp} shows that in the case that $Z(t) = \eta$, it suffices to deduce a pointwise-in-space error estimate for an approximation of $Y$ in order to derive a corresponding estimate for $X$. We turn to this setting now, and assume in addition that $\cA = a\Delta$ is equipped with Dirichlet zero boundary conditions at $x = 0$.
	
	The approximation scheme we use for $Y$ below is semi-implicit in time. This means that we must formulate it on a finite domain $(0,D)$ for sufficiently large $D$. From~\eqref{eq:X-closed-form-approximation} we see that we must at least take $D \ge 2 T - k $ if we want to evaluate $X$ up to $t = T$. We restrict the initial condition and the noise using the restriction operator $R_{\R^+ \to (0,D)}$ and introduce an artificial boundary condition at $x=D$. This causes a so called localization error. Such errors were recently investigated in detail in the Walsh setting in \cite{C22}, considering localization of stochastic heat equations on $\R$ to $(0,D)$. We use similar arguments for the semigroup framework of~\cite{DPZ14} that we have adopted here.
	
	We let $\cA_D := a\Delta$ denote the realization in $L^2{(0,D)}$ of the Laplacian with Dirichlet zero boundary conditions at $x = 0$ and $x = D$. The analytic semigroup $\cS_D = (\cS_D(t))_{t \ge 0}$ of contractions generated by $\cA_D$ maps $L^2(0,D)$ into a space of continuous and bounded functions on $[0,D]$. It has an explicit representation by
	\begin{equation*}
		(\cS_D(t) v)(x) = \int^D_0 \Phi_D(x,y,t) v(y) \dd y
	\end{equation*}
	for $x \in [0,D], t > 0$ and $v \in L^2((0,D))$, where
	\begin{equation*}
		\Phi_D(x,y,t) := \sum^\infty_{j=1} e_{D,j}(x) e_{D,j}(y) e^{- \lambda_{D,j} t},
	\end{equation*}
	with $\lambda_{D,j} := a\pi^2 j^2 / D^2$ and $e_{D,j}(x):=\sqrt{2/D}\sin(\pi j x /D)$ for $j \in \N$, $x \in [0,D]$.
	Directly from this representation, we obtain the rescaling property
	\begin{equation}
		\label{eq:rescale-semigroup}
		(\cS_D(t)v)(x) = \big(\cS_1(t/D^2)v(\cdot/D)\big)(x/D), \, x \in [0,D], t > 0.
	\end{equation}
	Using the fact that the sine functions $x \mapsto \sin(\pi j x)$ are uniformly bounded in $L^\infty$ with respect to $j$, this can be used to see that there is a constant $C < \infty$  such that for all $D > 0$, $v \in L^\infty((0,D))$ and $0 < s \le t < \infty$,
	\begin{equation}
		\label{eq:D-heat-holder-singularity}
		\norm[L^\infty]{(\cS_D(t) - \cS_D(s))v} \le s^{-1} (t-s) \norm[L^\infty]{v}.
	\end{equation} 
	Moreover, for all $D > 0$ and $t < \infty$, $\cS_D$ is stable in the sense that
	\begin{equation*}
		\norm[L^\infty]{\cS_D(t)v} \le \norm[L^\infty]{v}.
	\end{equation*}

	We now let $Y_D$ be the mild solution
	\begin{equation*}
		Y_D(t) = \cS_D(t)Y_D(0) + \int^{t}_0 \cS_D(t-s) R_{\R^+ \to (0,D)} \dd W(s), \, t \in [0,T],
	\end{equation*}to 
	\begin{equation*}
		\dd Y_D(t) = \cA_D Y_D (t) \dd t + R_{\R^+ \to (0,D)} \dd W(t), \, t \in [0,T],
	\end{equation*}
	with initial condition $Y_D(0) = R_{\R^+ \to (0,D)} Y(0)$.
	Fractional powers $(-\cA_D)^{r/2}$, $r \ge 0$ are well-defined since $-\cA_D$ is self-adjoint and positive definite. They are explicitly given by 
	\begin{equation*}
		(-\cA_D)^{r/2} v := \sum_{j = 1}^\infty (\lambda_{D,j})^{r/2} \inpro[L^2((0,D))]{v}{e_{D,j}} e_{D,j}
	\end{equation*}
	in terms of the eigenvalues $(\lambda_{D,j})_{j=1}^\infty$ and eigenfunctions $(e_{D,j})_{j=1}^\infty$ of $-\cA_D$, cf.\ \cite[Appendix~B]{K14}.
	This is well-defined for $v \in L^2$ such that the sequence on the right hand side is summable. This space is denoted by $\dot{H}_D^r := \dom((-\cA_D)^{r/2})$ and by another appeal to \cite[Theorem~1]{F67},
	\begin{equation}
		\label{eq:sobolev_id_D}
		\dot{H}_D^r = H^r((0,D)) \text{ if } r \in [0,1/2)
	\end{equation}
	with equivalent norms. 
	
	Under Assumption~\ref{ass:noise-kernel}\ref{ass:noise-kernel-iii}, the ideal property of $\cL_2(H_{q_W},\dot{H}_D^r)$ yields that for $r < \rho := \min(r_W,1/2)$
	\begin{equation}
		\label{eq:heat-finite-domain-well-posed-hs-estimate}
		\begin{split}
			\hspace{2em}\norm[\cL_2(H_{q_W},\dot{H}_D^r)]{R_{\R^+ \to (0,D)}} \le &\norm[\cL(H^r((0,D)),\dot{H}_D^r)]{I_{H^r((0,D)) \hookrightarrow \dot{H}_D^r}} \norm[\cL(H^r(\R^+),H^r((0,D)))]{R_{\R^+ \to (0,D)}} \\
			&\quad\times
			\norm[\cL_2(H_{q_W},H^{r}(\R^+))]{I_{H_{q_W} \hookrightarrow H^{r}(\R^+)}}.
			\raisetag{-1.5\baselineskip}
		\end{split}
	\end{equation}
	The setting of Theorem~\ref{thm:heat-regularity} therefore suffices to ensure that the stochastic convolution $t \mapsto Y_D(t) - \cS_D(t)Y_D(0)$ has a continuous modification in $\dot{H}^r$ for all $r < \rho + 1$. This may not hold for the entire process $t \mapsto Y_D(t)$ since $Y(0) \in \dot{H}^{\rho+1}$ does not ensure that $Y_D(0) \in \dot{H}_D^{\rho + 1}$. However, $Y(0) \in \dot{H}^{\rho+1}$ does guarantee that $Y_D(0)$ is bounded and Hölder continuous with exponent $r < \rho + 1/2$. We use these properties to deduce the following pointwise-in-space localization error between $Y$ and $Y_D$.

	\begin{proposition}
		\label{prop:localization-heat}
		Let the assumptions of Theorem~\ref{thm:heat-regularity} along with Assumption~\ref{ass:noise-kernel}\ref{ass:noise-kernel-iv} be satisfied. Then, there is a constant $C<\infty$ such that for all $D \ge 1$, $t \in [0,T]$ and $x \in [0,D]$, 
		\begin{align*}
			\bigE{\left|Y(t,x) - Y_D(t,x)\right|^2}^{1/2} \le C e^{-(D-x)^2/(8 a T)}.
		\end{align*}
	\end{proposition}
	
	\begin{proof}
		We assume for simplicity that $Y(0) = 0$ since the terms involving $Y(0)$ and $Y_D(0)$ can be estimated by similar arguments as below. 
		First, pointwise evaluation of $Y_D$ is admissible since the process
		takes values in $\dot{H}^{r}_D$ for all $r < \min(r_W,1/2) + 1$ by~\eqref{eq:heat-finite-domain-well-posed-hs-estimate}. Therefore, by the It\^o isometry, we have for any $x \in [0,D]$, 
		\begin{equation}
			\label{eq:prop:localization-heat:1}
			\begin{split}
				\bigE{\left|Y(t,x) - Y_D(t,x)\right|^2} &= \int^t_0 \norm[\cL_2(H_{q_W},\R)]{\delta_x \cS(t-s) - \delta_x \cS_D(t-s) R_{\R^+ \to (0,D)}}^2 \dd s \\
				&= \sum_{j = 1}^\infty \int^t_0 \left|\delta_x \cS(t-s)e_j - \delta_x \cS_D(t-s) R_{\R^+ \to (0,D)}e_j\right|^2 \dd s\\
				&= \sum_{j = 1}^\infty \int^t_0 \left|\delta_x \cS(s)e_j - \delta_x \cS_D(s) R_{\R^+ \to (0,D)}e_j\right|^2 \dd s
			\end{split}
		\end{equation}
		where $(e_j)_{j=1}^\infty$ is an arbitrary ONB of $H_{q_W}$. We choose the ONB in Assumption~\ref{ass:noise-kernel}\ref{ass:noise-kernel-iv}. For $j \in \N$, we use the kernel~\eqref{eq:halfline-heat-kernel} and make the split 
		\begin{align*}
			&\delta_x \cS(s)e_j - \delta_x \cS_D(s) R_{\R^+ \to (0,D)}e_j \\
			&\quad= \int_0^\infty \left(\Phi(x-y,s) - \Phi(x+y,s)\right) e_j(y) \dd y - \int_0^D \Phi_D(x,y,s) e_j(y) \dd y \\
			&\quad= \int_D^\infty \left(\Phi(x-y,s) - \Phi(x+y,s)\right) e_j(y) \dd y \\ &\qquad+ \int_0^D \left(\Phi(x-y,s) - \Phi(x+y,s) - \Phi_D(x,y,s)\right) e_j(y) \dd y =: \mathrm{I} + \mathrm{II}. 
		\end{align*}
		We note that
		\begin{equation*}
			\int^\infty_x \Phi(y,t) \dd y \le e^{-x^2/(4 a t)}
		\end{equation*}
		for $x \ge 0$ and $t > 0$. Applying this to the first term, it is bounded by
		\begin{equation*}
			|\mathrm{I}| \le \norm[L^\infty]{e_j} \left(\int^\infty_{D-x} \Phi(y,s) \dd y +  \int^\infty_{D+x} \Phi(y,s) \dd y \right) \le \norm[L^\infty]{e_j} \big(e^{-(D+x)^2/(4 a s)} + e^{-(D-x)^2/(4 a s)}\big) 
		\end{equation*}
		which in turn can be bounded by $2 \norm[L^\infty]{e_j} e^{-{(D-x)^2}/{(4 a T)}}$. For the term $\mathrm{II}$, we use
		the fact that
		\begin{equation*}
			\Phi_D(x,y,t) = \sum_{n=-\infty}^\infty \Phi(x+y + 2 n D,t) - \Phi(x-y + 2 n D,t), \, t > 0, x,y \in (0,D),
		\end{equation*}
		to rewrite it as 
		\begin{align*}
			\mathrm{II}&=\sum_{\substack{ n \in \Z \\ n \neq 0 }} \int_0^D \left(\Phi(x+y + 2 n D,s) - \Phi(x-y + 2 n D,s)\right) e_j(y) \dd y \\
			&= \sum_{\substack{ n \in \Z \\ n \neq 0 }} \int_{x + 2nD}^{x + (2 n+1)D} \Phi(y,s) e_j(y-x-2 n D) \dd y - \int_{-x-2nD}^{-x-(2 n -1)D} \Phi(y,s) e_j(y+x+2 n D) \dd y.
		\end{align*}
		Using the non-negativity and symmetry of $\Phi$, it is bounded by
		\begin{align*}
			|\mathrm{II}| \le 2 \norm[L^\infty]{e_j} \sum_{\substack{ n \in \Z \\ n \neq 0 }} \int^\infty_{x+2D} \Phi(y,t) \dd y \le 2 \norm[L^\infty]{e_j} e^{-(D+x)^2/(4 a s)} \le 2 \norm[L^\infty]{e_j} e^{-(D-x)^2/(4 a T)}.
		\end{align*}
		These estimates along with Assumption~\ref{ass:noise-kernel}\ref{ass:noise-kernel-iv} in~\eqref{eq:prop:localization-heat:1} finish the proof.
	\end{proof}
	
	We now introduce a fully discrete approximation of $Y_D$ based on finite elements and the backward Euler scheme. For this we consider a spatial $(x_j)_{j = 0, \ldots, N_h}$ as well as a temporal grid $(t_i)_{i = 0, \ldots, N_k}$ for $h, k \in (0,1]$. As before $x_j = j h$ and $t_i = i k$,  and we assume for simplicity that $N_h = D/h \in \N$ and $N_k = T/k \in \N$. We let $V_h$ be the space of piecewise linear functions on the spatial grid and $V_h^0$ the subspace of functions that are zero at $x = 0$ and $x = D$. Both are finite dimensional Hilbert subspaces of $L^2((0,D))$ when equipped with the topology of this space. The latter is a subspace of $H^1_0(D) := \{v \in H^1 : v(0) = v(1) = 0\}$. We define a discrete operator $\cA_h$ on $V_h^0$ by
	\begin{equation*}
		\inpro[L^2]{\cA_h u_h}{v_h} := -a\inpro[L^2]{u_h'}{v_h'}, \, u_h, v_h \in V_h^0.
	\end{equation*}
	We write $P^0_h: L^2 \to V^0_h$ for the orthogonal $L^2$-projection onto $V^0_h$ and define the interpolant $I_h: H^r \to V_h$ by $(I_h v)(x_j) := v(x_j)$ for $j \in \{0,\ldots,N_h\}$. The latter is bounded on $H^r((0,D))$ for $r > 1/2$.
	A discretization of $Y_D$ is now given by the $V^0_h$-valued sequence $Y^{h,k}_D = (Y^{h,k}_D(t_i))_{i = 0, \ldots, N_k}$ which is defined by $Y^{h,k}_D(0) = P^0_h Y_D(0)$ and 
	\begin{equation}
		\label{eq:fem-recursion}
		(I - k \cA_h) Y^{h,k}_D(t_{i+1}) = Y^{h,k}_D(t_{i}) + P^0_h I_h R_{\R^+ \to (0,D)}\Delta W^i, \, i = 0, 1, \ldots, N_k-1,
	\end{equation}
	where $I$ denotes the identity on $V^0_h$, otherwise. 
	
	The last term in~\eqref{eq:fem-recursion} is well-defined since $\delta_x$ is bounded on $H_{q_W}$ for all $x \in \R^+$. Therefore $I_h$ is well-defined on $H_{q_W}$ and
	\begin{equation*}
		\E[\norm[V_h]{I_h R_{\R^+ \to (0,D)} \Delta W^i}^2] = \E\left[\lrnorm[V_h]{\int^{t_{i+1}}_{t_i}I_h R_{\R^+ \to (0,D)}\dd W(s)}^2\right] = \norm[\cL_2(H_{q_W}((0,D)),V_h)]{I_h}^2 < \infty,
	\end{equation*}
	since $I_h$ has finite-dimensional range. Note that, since $\delta_x$ is bounded on $V_h$, we have similarly
	\begin{equation*}
		\Cov((I_h \Delta W^i)(x_m),(I_h \Delta W^i)(x_n)) = k \sum_{j = 1}^\infty \inpro[\R]{\delta_{x_m}I_h R_{\R^+ \to (0,D)} e_j}{\delta_{x_m}I_h R_{\R^+ \to (0,D)} e_j}
	\end{equation*}
	where $(e_j)_{j=1}^\infty$ is an ONB of $H_{q_W}(\R^+)$. Since $(R_{\R^+ \to (0,D)} e_j)_{j=1}^\infty$ is an ONB of $H_{q_W}((0,D))$, we obtain 
	\begin{equation*}
		\Cov((I_h \Delta W^i)(x_m),(I_h \Delta W^i)(x_n)) = k q_W(x_m,x_n).
	\end{equation*}
	Therefore, $I_h \Delta W^i$ may be obtained by sampling a Gaussian random vector in $\R^{N_h+1}$ with covariance matrix $\mathbf{C}$, having entries $\mathbf{C}_{i,j} = k q_W(x_i,x_j)$. For (weighted) Mat\'ern covariance functions this can be accomplished in log-linear complexity with respect to $N_h$ by using a circulant embedding algorithm (see~\cite{GKNSS18}) and multiplying by the weight function.
	
	We define an approximation $\cS_D^{h,k} = (\cS_D^{h,k}(t))_{t \ge 0}$ of the semigroup $\cS_D$ by 
	\begin{equation*}
		\cS_D^{h,k}(t) :=  \I_{\{0\}}(t) I + \sum^{\infty}_{i=0}  \I_{(t_i,t_{i+1}]}(t) (I - k \cA_h)^{i+1}.
	\end{equation*}
	where $\I$ is the indicator function. Here we have extended the temporal grid by $t_i = i k $ for $i \in \N$. This allows us to write~\eqref{eq:fem-recursion} in closed form by
	\begin{equation}
		\label{eq:fem-mild}
		Y^{h,k}_D(t_{n}) = \cS_D^{h,k}(t_n) P_h^0 Y_D(0) + \int^{t_n}_0 \cS_D^{h,k}(t_n-s) P^0_h I_h R_{\R^+ \to (0,D)} \dd W(s)
	\end{equation}
	for $t_n \le T$.
	Like $\cS_D$, this operator family has a scaling property
	\begin{equation}
		\label{eq:rescale-fd-semigroup}
		(\cS_D^{h,k}(t)v_h)(x) = \big(\cS_1^{h/D,k/D^2}(t/D^2)v(\cdot/D)\big)(x/D), \, x \in [0,D], t > 0.
	\end{equation}
	
	In order to derive a a pointwise-in-space error bound for $Y_D(t_{i})-Y^{h,k}_D(t_{i})$, we need some stability and error estimates for the finite element method. We collect them in the following lemma.
	
	\begin{lemma}
		\label{lem:linfty-estimates-finite-elements}
		In the setting outlined above, the following estimates hold:
		\begin{enumerate}[label=(\roman*)]
			\item \label{eq:projection-stability} $\sup_{\substack{h \in (0,1] \\ D > 0}} \norm[\cL(L^\infty((0,D)))]{P_h^0} < \infty,$
			\item \label{eq:interpolation-error}
			$\sup_{\substack{h \in (0,1] \\ D > 0}} h^{-\min(s,2)}\norm[\cL(H^s((0,D)),L^2((0,D)))]{(I-I_h)} < \infty$, $s \in (1/2,2],$
			\item\label{eq:semigroup-error-bound} $\sup_{\substack{h,k \in (0,1] \\ D, t > 0}}
			t^s (h^{2 s}  + k^s)^{-1} \norm[\cL(L^\infty((0,D)))]{\cS_D(t) - \cS_D^{h,k}(t) P_h^0} < \infty,$ $s \in [0,1]$, and
			\item \label{eq:disc-semigroup-singularity}
			$\sup_{\substack{h,k \in (0,1] \\ D, t > 0}} t^{r/2} \norm[\cL(L^2((0,D)))]{(-\cA_h)^{r/2} \cS_D^{h,k}(t) P_h^0 v} < \infty$.
		\end{enumerate}
	\end{lemma}
	
	\begin{proof}
		The stability estimate~\ref{eq:projection-stability} is \cite[Theorem~1]{CT87}. This is stated for $D = 1$, but the proof holds also for general $D$ with no change in the bound. 
		
		A proof of the interpolation error~\ref{eq:interpolation-error} is given in~\cite[Example~3]{DS80} for domains $\cD \in \R^2$ and $r > 1$, but the argument is the same in our case, noting that $C$ does not depend on the size of $\cD$. The restriction $r>1$ is required in $\R^2$ for $I_h$ to be well-defined on $H^r(\cD)$, but $r > 1/2$ suffices in our case. The proof uses the Sobolev--Slobodeckij norm, but this is bounded by the norm of $H^r((0,D))$ and, crucially, the bound is uniform in $D$. 
		
		For the semigroup error bound~\ref{eq:semigroup-error-bound}, we first use~\ref{eq:projection-stability} along with~\cite[(1.17)]{CLT94} to obtain a constant $C < \infty$ such that for all $v \in L^\infty((0,1))$, $h,k \in (0,1]$ and $n \in \N$, $\norm[L^\infty]{\cS_1^{h,k}(t_n) P_h^0 v} \le C \norm[L^\infty]{v}$. By~\cite[(1.14) and (1.21)]{CLT94}, there is a constant $C < \infty$ such that for all $v \in L^\infty((0,1))$,  $n \in \N$, and $h,k \in (0,1]$, $$\norm[L^\infty]{(\cS_1(t_n) - \cS_1^{h,k}(t_n) P_h^0) v} \le C t_n^{-1}(h^2  + k )\norm[L^\infty]{v}.$$
		Combining these estimates with~\eqref{eq:rescale-semigroup}, \eqref{eq:D-heat-holder-singularity} and \eqref{eq:rescale-fd-semigroup} yields
		the existence of a constant $C > 0$ such that for all $t > 0$, $D > 0$, $h,k \in (0,1]$ and $v \in L^\infty((0,D))$, 
		$$\norm[L^\infty]{(\cS_D(t) - \cS_D^{h,k}(t) P_h^0) v} \le C t^{-1}( h^2  + k )\norm[L^\infty]{v}.$$
		The same rescaling argument yields stability of $\cS_D^{h,k}$ for general $D$, whence we obtain~\ref{eq:semigroup-error-bound}.
		
		Finally, the proof of~\ref{eq:disc-semigroup-singularity} is given in~\cite[Lemma~7.3]{T06} along with an interpolation argument (see, e.g., \cite[Theorem~A.4]{BZ00} for a justification), noting that the estimate of~\cite{T06} does not depend on the specific semigroup.
	\end{proof}

	We are now ready to prove a pointwise-in-space error bound for $Y_D(t_{i})-Y^{h,k}_D(t_{i})$.
	\begin{proposition}
		\label{prop:fem-convergence}
		Let the assumptions of Theorem~\ref{thm:heat-regularity} along with Assumption~\ref{ass:noise-kernel}\ref{ass:noise-kernel-iv} and Assumption~\ref{ass:noise-kernel-w} be satisfied. Then, for all $r_1 \in [0,1], r_2 \in [0,1/2)$ and $r_3 \in [0,\min(s_{W},2))$, there is a constant $C < \infty$ such that for all $D \ge 1$, $h,k \in (0,1]$, $i = 1, \ldots, N_k$, and $x \in [0,D]$,
		\begin{equation*}
			\E\left[ |Y_D(t_i,x) - Y_D^{h,k}(t_i,x)|^2 \right]^{1/2} \le C \left(t_i^{-r_1}(h^{2 r_1} + k^{r_1}) + h^{2r_2} + k^{r_2} + D h^{r_3} \right).
		\end{equation*}
	\end{proposition}
	
	\begin{proof}
		In an argument similar to the proof of Proposition~\ref{prop:holder}, we again use the fact that $\delta_x$ is continuous on $H^r((0,D))$ for any $r > 1/2$ when applying the It\^o isometry and then bound the integrands by the $L^\infty$ norm. With the mild formulas~\eqref{eq:mild-stochastic-heat} and~\eqref{eq:fem-mild}, we split the error into three parts, 
		\begin{align*}
			\E\left[ |Y_D(t_i,x) - Y_D^{h,k}(t_i,x)|^2 \right] &\lesssim \norm[L^\infty]{(\cS_D(t_i) - \cS_D^{h,k}(t_i) P_h^0) Y_D(0)}^2 \\
			&\quad+ \sum_{j=1}^\infty \int^{t_i}_0 \norm[L^\infty]{(\cS_D(t_i-s) - \cS_D^{h,k}(t_i-s) P_h^0) e_j}^2 \dd s \\
			&\quad+ \sum_{j=1}^\infty \int^{t_i}_0 \norm[L^\infty]{\cS_D^{h,k}(t_i-s) P_h^0 (I-I_h)R_{\R^+ \to (0,D)} e_j}^2 \dd s
		\end{align*}
		for any $x \in (0,D)$. Again, the ONB $(e_j)_{j=1}^\infty$ of $H_{q_W} = H_{q_W}(\R^+)$ is arbitrary. The first term on the right hand side is handled directly by~Lemma~\ref{lem:linfty-estimates-finite-elements}\ref{eq:semigroup-error-bound}. 
		
		For the second term, we choose the ONB of Assumption~\ref{ass:noise-kernel}\ref{ass:noise-kernel-iv} and use~Lemma~\ref{lem:linfty-estimates-finite-elements}\ref{eq:semigroup-error-bound} to see that it is bounded by a constant times 
		\begin{equation*}
			\sum_{j=1}^\infty \norm[L^\infty]{e_j}^2 (h^{2 r_2} + k^{r_2})^2 \int_{0}^{\infty} (t_i - s)^{-2 {r_2}} \dd r_2 \lesssim (h^{2 r_2} + k^{r_2})^2
		\end{equation*}
		for all $r_2 < 1/2$.
		
		For the third term we use the Poincar\'e inequality $$\norm[H^1((0,D))]{v_h} \lesssim D \norm[H^1_0]{v_h} = D \norm[L^2((0,D))]{A_h^{1/2}v_h}$$ for $v_h \in V_h^0$ along with an interpolation argument and a Sobolev embedding to obtain that for all $\epsilon > 0$ it is bounded by
		\begin{equation*}
			D^{1+2\epsilon} \int^{t_i}_0 \norm[\cL_2(H_{q_W},L^2)]{(-\cA_h)^{(1/2+\epsilon)/2} \cS_D^{h,k}(t_i-s) P_h^0 (I-I_h)R_{\R^+ \to (0,D)}}^2 \dd s.
		\end{equation*}
		We use the identity~\eqref{eq:hs-norm-approximation-numbers} to obtain a bound on the integrand in terms of approximation numbers. In order to make use of~\eqref{eq:multiplicative-property}, the sum is split into two parts:
		\begin{equation}
			\label{eq:prop:fem-convergence:pf:1}
			\begin{split}
				&\norm[\cL_2(H_{q_W},L^2)]{(-\cA_h)^{(1/2+\epsilon)/2} \cS_D^{h,k}(t_i-s) P_h^0 (I-I_h)R_{\R^+ \to (0,D)}}^2 \\ &= \sum^\infty_{j = 1} \psi_{2j-1}\left((-\cA_h)^{(1/2+\epsilon)/2} \cS_D^{h,k}(t_i-s) P_h^0 (I-I_h)R_{\R^+ \to (0,D)}\right)^2 \\
				&\quad+ \sum^\infty_{j = 1} \psi_{2j}\left((-\cA_h)^{(1/2+\epsilon)/2} \cS_D^{h,k}(t_i-s) P_h^0 (I-I_h)R_{\R^+ \to (0,D)}\right)^2.
			\end{split}
		\end{equation}
		For the first sum, we use~\eqref{eq:operator-norm-property}, \eqref{eq:multiplicative-property} and \eqref{eq:approximation-bound-by-entropy} to see that for all ${r_3} < s_W$
		\begin{align*}
			&\psi_{2j - 1}\left((-\cA_h)^{(1/2+\epsilon)/2} \cS_D^{h,k}(t_i-s) P_h^0 (I-I_h)R_{\R^+ \to (0,D)}\right) \\ &\quad\le 2 \psi_j\left((-\cA_h)^{(-1/2+2\epsilon)/2} \right) \norm[\cL(L^2)]{(-\cA_h)^{(1-\epsilon)/2} \cS_D^{h,k}(t_i-s) P_h^0} \norm[\cL(H^{r_3}, L^2)]{I-I_h} \\ &\qquad\times\norm[\cL(H^{r_3}(\R),H^{r_3}((0,D)))]{R_{\R \to (0,D)}} \epsilon_j(I_{H^{s_W,\alpha} \hookrightarrow H^{r_3}} ) \norm[\cL(H_{q_W}(\R),H^{s_W,\alpha}(\R))]{I_{H_{q_W} \hookrightarrow H^{s_W,\alpha}}} \\
			&\qquad\times \norm[\cL(H_{q_W}(\R^+),H_{q_W}(\R))]{E_{\R^+ \to \R}}.
		\end{align*}
		By Lemma~\ref{lem:linfty-estimates-finite-elements}\ref{eq:interpolation-error} and \ref{eq:disc-semigroup-singularity} along with H\"older's inequality for sequence spaces, therefore, it is bounded by a constant times 
		\begin{equation*}
			(t_i-s)^{\epsilon - 1} h^{2 \min({r_3},2)} \lrnorm[\ell^p]{\left(\psi_j\left((-\cA_h)^{(-1/2+2\epsilon)/2} \right)\right)_{j=1}^\infty}^{2}
			\lrnorm[\ell^q]{\left(\epsilon_j(I_{H^{s_W,\alpha} \hookrightarrow H^{r_3}} )\right)_{j=1}^\infty}^{2}
		\end{equation*}
		for all $p,q > 2$ such that $1/p + 1/q = 1/2$. As in the proof of Theorem~\ref{thm:noise-kernel}, for sufficiently small $s_W - {r_3} > 0$, $\epsilon_j(I_{H^{s_W,\alpha} \hookrightarrow H^{r_3}} ) \le j^{{r_3}-s_W}$. Letting $q > 1/(s_W-{r_3})$ yields $\left(\epsilon_j(I_{H^{s_W,\alpha} \hookrightarrow H^{r_3}} )\right)_{j=1}^\infty \in \ell^q$. By the min-max principle, the eigenvalues of $-\cA_h \in \cL(V^0_h)$ are bounded from below by the eigenvalues of $-\cA$. Since approximation numbers coincide with eigenvalues for self-adjoint operators on Hilbert spaces (see~\cite[Proposition~11.3.3]{P80}), it follows that 
		\begin{align*}
			\lrnorm[\ell^p]{\left(\psi_j\left((-\cA_h)^{(-1/2+2\epsilon)/2} \right)\right)_{j=1}^\infty}^{2} &\le \left(\sum_{j = 1}^\infty (a\pi^2 j^2 / D^2)^{p(-1/2+2\epsilon)/2}\right)^{2/p} \\&= a^{2\epsilon - 1/2}D^{1-4\epsilon} \pi^{4 \epsilon - 1} \left(\sum_{j = 1}^\infty j^{2\epsilon p-p/2}\right)^{2/p}. %
		\end{align*}
		Since $p >2$, this sum is finite as long as $\epsilon > 0$ is sufficiently small. The second sum in~\eqref{eq:prop:fem-convergence:pf:1} is handled in the same way, so that 
		\begin{equation*}
			\sum_{j=1}^\infty \int^{t_i}_0 \norm[L^\infty]{\cS_D^{h,k}(t_i-s) P_h^0 (I-I_h)R_{\R^+ \to (0,D)} e_j}^2 \dd s \lesssim D^2 h^{2\min({r_3},2)}. \qedhere
		\end{equation*}
	\end{proof}
	
	\begin{remark}
		The singularity $t_i^{-r_1}$ may be removed by assuming sufficient regularity of $Y_D$, cf.\ \cite[Theorem~3.6]{CLT94}. We do not discuss this further other than noting that this requires $Y(0,x) = 0$ for $x = D$. In practice this entails assuming that $Y(0)$ has compact support since we employ $Y_D$ as an approximation of $Y$ with increasing $D$.
	\end{remark}
	
	\begin{remark}
		\label{rem:weight-stationary-pointwise-discussion}
		Suppose that $Y(0) = 0$ and that $q_W$ is given by a weight-stationary kernel as in Theorem~\ref{thm:noise-kernel} with smoothness parameter $\sigma > 1/2$ and weight parameter $\alpha_W > 1/4$. Taking also Lemma~\ref{lem:noise-kernel} into account, the result of Proposition~\ref{prop:fem-convergence} reads: there is for all $r_2 \le 1/2$ and $r_3 \le \sigma$ a constant $C < \infty$ such that for all $D \ge 1$, $h,k \in (0,1]$, $i = 1, \ldots, N_k$, and $x \in [0,D]$,
		\begin{equation*}
			\E\left[ |Y_D(t_i,x) - Y_D^{h,k}(t_i,x)|^2 \right]^{1/2} \le C \left(h^{2r_2} + k^{r_2} + D h^{r_3} \right).
		\end{equation*}
		We see that the rate in time is essentially $1/2$ no matter the value of $\sigma$ while the rate in space varies between $1/2$ and $1$. If we had sampled the noise by $P_h^0 R_{\R^+ \to (0,D)} \Delta W$ directly in~\eqref{eq:fem-recursion} instead of employing $I_h$ we would see a convergence rate of $1$ in space no matter the value of $\sigma$. A covariance matrix of the elements in the vector corresponding to~$P_h^0 R_{\R^+ \to (0,D)} \Delta W$ can be computed and used for this purpose. There are two problems with this approach, however. First, the matrix involves integrals that must in general be computed numerically. This can be very expensive. Second, even for weighted Mat\'ern kernels, there are no theoretical guarantees that circulant embedding algorithms can be employed. Then more expensive methods, such as using a Cholesky decompositions and multiplying with the resulting dense matrix for sampling, must be employed.
	\end{remark}
	
	\begin{remark}
		\label{rem:sobolev-embedding}
		In the proof of Proposition~\ref{prop:fem-convergence}, we used a Sobolev embedding technique for the term involving $I_h$. Had we also used it for the other integral term, we could employ classical techniques (see, e.g., \cite[Theorem~1.1]{Y05}) combined with interpolation between error estimates in $\cS_D(t) - \cS_D^{h,k}(t) P_h^0$ measured in $L^2$ and $\dot{H}_D^1$ norms (cf.\ \cite[Lemma~3.12]{K14}) to obtain a convergence rate of essentially $h^{1/2} + k^{1/4}$. That is, assuming that no essential smoothness of the kernel in the Hilbert--Schmidt sense was available. This is the case if a weight-stationary kernel as in Theorem~\ref{thm:noise-kernel} with $\sigma = 1/2 + \epsilon$ and $\epsilon > 0$ small is used. From the previous remark, we see that this rate would be suboptimal in time. It would also be suboptimal in space if we were to sample the noise by $P_h^0 R_{\R^+ \to (0,D)} \Delta W$ directly.
	\end{remark}
	
	Traditionally, finite element approximations of the stochastic heat equation have been analyzed in $L^2$ norms in space, see, e.g., \cite{Y05,K14}. In the same setting as in Remark~\ref{rem:weight-stationary-pointwise-discussion}, suppose also for simplicity that $D = 1$. Then, an application of \cite[Lemma~3.12]{K12} and Lemma~\ref{lem:linfty-estimates-finite-elements}\ref{eq:interpolation-error} yields for all $r_2 \in [0,1/2 + \min(\sigma-1/2,1/2)/2)$ and $r_3 \in [0,\sigma)$, the existence of a constant $C < \infty$ such that for all $h,k \in (0,1]$, $i = 1, \ldots, N_k$,
	\begin{equation*}
		\E\left[ \norm[L^2]{Y_D(t_i) - Y_D^{h,k}(t_i)}^2 \right]^{1/2} \le C(h^{2 r_2} + k^{r_2} + h^{r_3}).
	\end{equation*}
	Compared to the results of Proposition~\ref{prop:fem-convergence}, we see that if $\sigma$ is close to $1/2$, the same rate is obtained. Otherwise, the rate can be as high as $3/2$ in space and $3/4$ in time. 	The reason for the potentially higher rate is that Assumption~\ref{ass:noise-kernel}\ref{ass:noise-kernel-iii} may be used with $r_W < \sigma - 1/2$ combined with~\eqref{eq:sobolev_id_D}. Since no such equivalence is available in our $L^\infty$ setting, a lower rate is obtained. This is however sharp, which we now demonstrate in simulations.
	
	\begin{figure}[ht!]
		\centering
		\includegraphics[width = \textwidth]{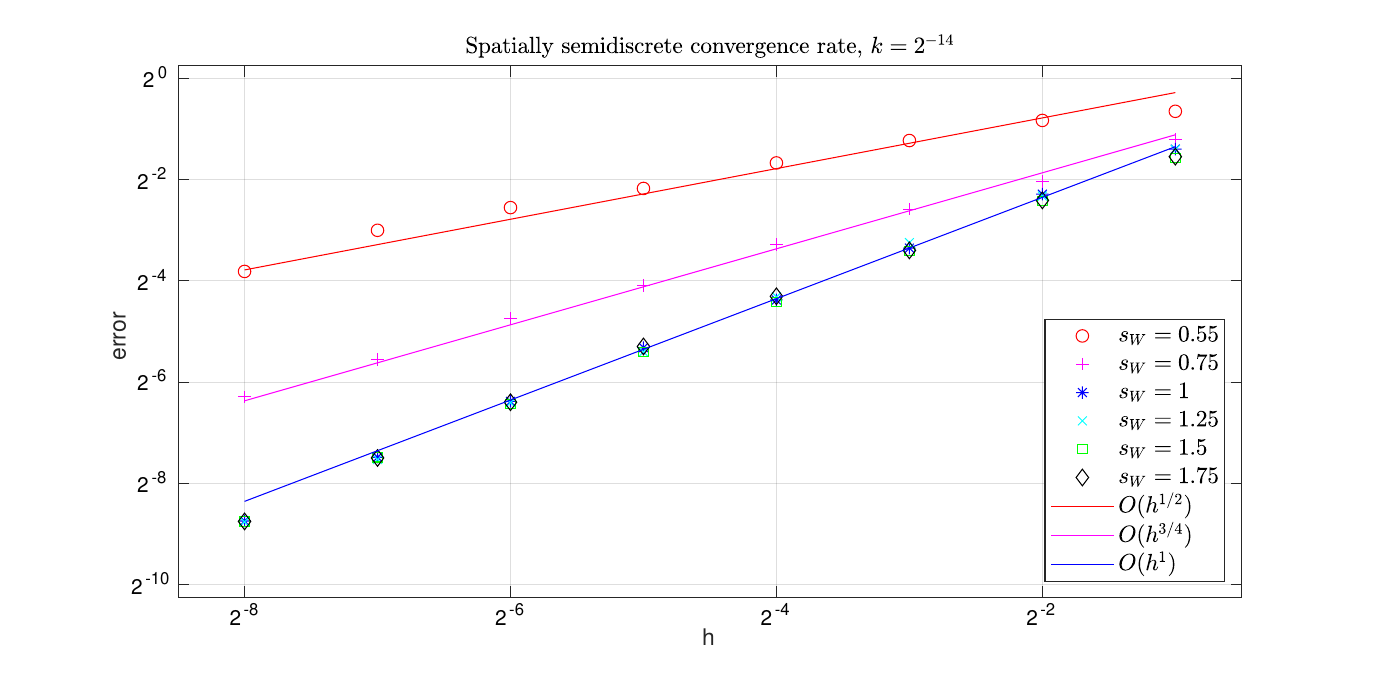}
		\caption{Errors $\max_{t_i,x} \E[|Y_D(t_i,x) - Y^{k,h}_D(t_i,x)|^2]^{1/2}$ for the discretized and localized stochastic heat equation and decreasing space step sizes $h$ with time step $k$ fixed and small.}
		\label{fig:spatially-semidiscrete-errors}
	\end{figure}
	
	\begin{figure}[ht!]
		\centering
		\includegraphics[width = \textwidth]{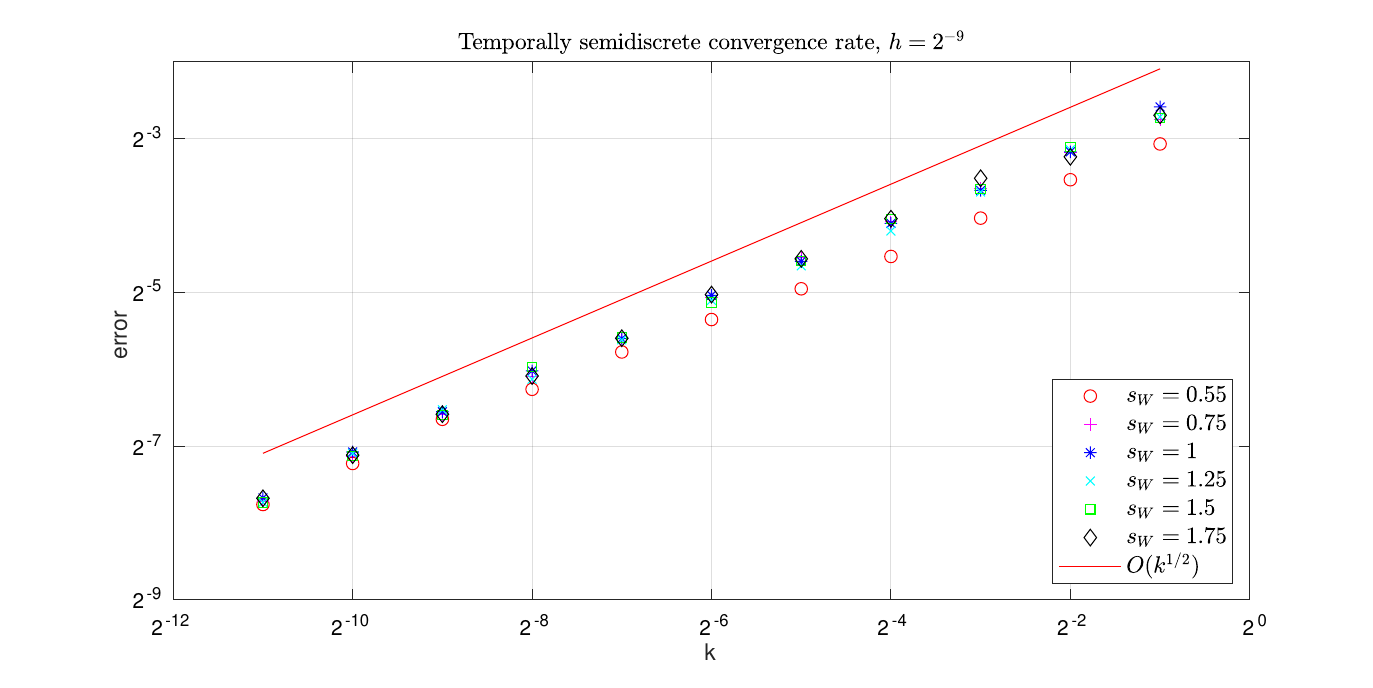}
		\caption{Errors $\max_{t_i,x}\E[|Y_D(t_i,x) - Y^{k,h}_D(t_i,x)|^2]^{1/2}$ for the discretized and localized stochastic heat equation and decreasing time step sizes $k$ with space step $h$ fixed and small.}
		\label{fig:temporally-semidiscrete-errors}
	\end{figure}
	
	We let $D = T := 1$, $Y_D := 0$ and $ a := 0.05$. We let the kernel $q_W$ of $W$ be of Mat\'ern class (see Remark~\ref{rem:matern}) with $\mu = \zeta := 1$ and increasing $\nu > 0$. This may be regarded as a weighted stationary kernel in Theorem~\ref{thm:noise-kernel} by choosing the weight function as a smooth bump function with large but finite support and equal to $1$ in a sufficiently large neighborhood of $0$. In this context, the results of Proposition~\ref{prop:fem-convergence} hold true with $s_W := \nu + 1/2$. We plot approximations of the errors
	\begin{equation}
		\label{eq:error-to-plot}
		\max_{\substack{i \in \{0,\ldots,N_k\} \\ x \in [0,D]}} \E[|Y_D(t_i,x) - Y^{k,h}_D(t_i,x)|^2]^{1/2}
	\end{equation} 
	based on Monte Carlo approximations with sample size $N = 100$. First, we fix $k = 2^{-14}$ and examine the spatial convergence rate for decreasing $h >0$. We replace $Y_D$ with a reference solution computed at $h = 2^{-10}$. As we see from Figure~\ref{fig:spatially-semidiscrete-errors}, the convergence rate increases from essentially $1/2$ when $s_W \approx 1/2$ to $1$ when $s_W \ge 1$. This is in line with the results of Theorem~\ref{thm:noise-kernel}. Next, we fix $h = 2^{-9}$ and examine the temporal convergence rate for decreasing $k > 0$. We replace $Y_D$ with a reference solution computed at $k = 2^{-14}$. We see from Figure~\ref{fig:temporally-semidiscrete-errors}, that the convergence rate stays constant no matter the value of $s_W$. This is again in line with the results of Theorem~\ref{thm:noise-kernel}.
	
	The sample size was chosen to ensure that the plots in Figures~\ref{fig:spatially-semidiscrete-errors} and~\ref{fig:temporally-semidiscrete-errors} were sufficiently stable. We used the same $N=100$ samples of $W$ across all values of $h$ for a given value of $s_W$. The samples across different $s_W$ were drawn independently of each other. The sample size $N=100$ does not yield a precise estimate of~\eqref{eq:error-to-plot}. However, we are only interested in observing the rate with which the error decrease as $h$ or $k$ tends to $0$, and in this sense the Monte Carlo error will never dominate, at least in a mean square sense and for fixed $t_i$, $x$. 
	For, with $E_{i,x} := |Y_D(t_i,x) - Y^{k,h}_D(t_i,x)|^2$, it follows from Proposition~\ref{prop:fem-convergence} that
	\begin{equation*}
		\E\Bigg[\Bigg| \Bigg( \frac{1}{N} \sum_{j = 1}^N E_{i,x}^{(j)}\Bigg)^{1/2} - \E[E_{i,x}]^{1/2} \Bigg|^2\Bigg] \le 4 \E\left[E_{i,x}\right] \lesssim \left(t_i^{-r_1}(h^{2 r_1} + k^{r_1}) + h^{2r_2} + k^{r_2} + D h^{r_3}\right)^{2},
	\end{equation*}
	where $(E_{i,x}^{(j)})_{j=1}^N$ are the independent samples of $E_{i,x}$ that make up the Monte Carlo estimate.
	
	\subsection{A fully discrete finite difference-finite element discretization of the HEIDIH model}
	\label{subsec:final-approx}
	
	We end this section on numerical approximation by combining the results of Propositions~\ref{prop:X-error-decomp}, \ref{prop:localization-heat} and~\ref{prop:fem-convergence}, thereby obtaining a fully discrete approximation result for the HEIDIH model in the case that $Z(t) := \eta$ for all $t \in [0,T]$.
	
	We employ the same finite difference method of Section~\ref{subsec:error-decomp} to define an approximation $X^{h,k}_D$ on a space-time grid with equal spatial and temporal step sizes $k \in (0,T]$.  Recalling~\eqref{eq:X-with-eta-simple-version}, this is defined by
	\begin{equation}
		\label{eq:fully-disc-recursion}
		X_D^{h,k}(t_{i+1},x_j) = X_D^{h,k}(t_{i},x_{j+1}) + \norm[H_{q_B}]{I^*_{H_{q_B} \hookrightarrow \cH^r} \eta} Y_D^{h,k}(t_{i},x_j) \Delta \tilde \beta^{i}
	\end{equation}
	for $t_i = i k \ge 0$ and $x_j = jk \ge 0$.
	Here $Y_D^{h,k}$ is the approximation of $Y_D$ defined by~\eqref{eq:fem-recursion} and $\tilde \beta$ the real-valued Wiener process defined in~\eqref{eq:beta-tilde} with $\Delta {\tilde \beta}^i = \tilde \beta(t_{i+1})-\tilde \beta(t_i)$. Note that $Y_D^{h,k}$ is computed on a space-time grid with temporal step size $k$ and spatial mesh size $h$ that may not be equal. Since $Y_D^{h,k}(t_i,\cdot)$ is a piecewise linear function on $[0,D]$, the scheme~\eqref{eq:fully-disc-recursion} is well-defined even if $h \neq k$.
	The closed form is given by
	\begin{equation}
		\label{eq:X-closed-form-approximation-eta}
		X_D^{h,k}(t_{n},x_j) = \delta_{x_j} \cS(t_{n}) X(0) + \norm[H_{q_B}]{I^*_{H_{q_B} \hookrightarrow \cH^r} \eta} \sum^{n-1}_{i=0} \int^{t_{i+1}}_{t_i}  Y_D^{h,k}(t_{i},x_j+t_n - t_{i+1})  \dd \tilde \beta(r).
	\end{equation}
	
	\begin{theorem}
		\label{thm:full-approximation}
		Let the conditions of Theorem~\ref{thm:transport-regularity} and Proposition~\ref{prop:fem-convergence} be satisfied. Then, for all $s_1 < 1/2, s_2 < \min(s_W,2)$ and $T \ge 1$ there is a constant $C < \infty$ such that for all $h,k \in (0,1]$ and $D \ge 2 T - k$,
		\begin{equation*}
			\max_{n,j \in \{0,\ldots,N_k\}} \E\left[ \left| X(t_n,x_j) - X^{h,k}_D(t_n,x_j) \right|^2 \right]^{1/2} \le C \left(k^{s_1} + h^{2 s_1} + D h^{s_2} + e^{-(D+k-2 T)^2/(8 a T)} \right).
		\end{equation*}
	\end{theorem}
	
	\begin{proof}
		While it might not be possible to extend $Y^{h,k}_D$ to $\cH^r$ if $r$ is big, the equations~\eqref{eq:X-with-eta-simple-version} and~\eqref{eq:X-closed-form-approximation-eta} are well-defined so that we may use the second error decomposition formula in Proposition~\ref{prop:X-error-decomp} and split the error by
		\begingroup
		\allowdisplaybreaks
		\begin{align*}
			&\bigE{\left|X(t_{n},x_{j}) - X^{h,k}_D(t_{n},x_{j})\right|^2} \\
			&\quad= \norm[H_{q_B}]{I^*_{H_{q_B} \hookrightarrow \cH^r} \eta}^2 \sum^{n-1}_{i=0} \int^{t_{i+1}}_{t_i} \bigE{\left|Y^{h,k}_D(t_i,t_{n} + x_{j}- t_{i+1})-Y(r,t_{n} + x_{j}-r)\right|^2} \dd r \\
			&\quad\lesssim \int^{k}_{0} \norm[L^\infty((0,D))]{(I-P^0_h)R_{\R^+\to(0,D)}Y(0)}^2 \dd r \\
			&\qquad+ \sum_{i=1}^{n-1} \int^{t_{i+1}}_{t_i} \E\left[ |Y_D^{h,k}(t_i,t_{n} + x_{j}- t_{i+1})-Y_D(t_i,t_{n} + x_{j}- t_{i+1})|^2 \right] \dd r \\
			&\qquad+ \sum_{i=0}^{n-1} \int^{t_{i+1}}_{t_i} \E\left[ |Y_D(t_i,t_{n} + x_{j}- t_{i+1})-Y(t_i,t_{n} + x_{j}- t_{i+1})|^2 \right] \dd r \\
			&\qquad+ \int^{k}_{0} \E\left[ |Y(0,t_{n} + x_{j}- t_{i+1})-Y(r,t_{n} + x_{j}- t_{i+1})|^2 \right] \dd r \\
			&\qquad+ \sum_{i=1}^{n-1} \int^{t_{i+1}}_{t_i} \E\left[ |Y(t_i,t_{n} + x_{j}- t_{i+1})-Y(r,t_{n} + x_{j}- t_{i+1})|^2 \right] \dd r \\
			&\qquad+ \sum_{i=0}^{n-1} \int^{t_{i+1}}_{t_i} \E\left[ |Y(r,t_n + x_j - t_{i+1})-Y(r,t_n + x_j - r|^2 \right] \dd r.
		\end{align*}%
		\endgroup
		The first and fourth term are bounded by a constant times $k$ by Lemma~\ref{lem:linfty-estimates-finite-elements}\ref{eq:projection-stability} along with Theorem~\ref{thm:heat-regularity} and a Sobolev embedding argument. The second is treated by Proposition~\ref{prop:fem-convergence} with $r_1 = r_2 = s_1$ and $r_3 = s_2$, the third by Proposition~\ref{prop:localization-heat} and the last two by Proposition~\ref{prop:holder}.
	\end{proof}
	
	\begin{remark}
		Note that this result provides a convergence estimate in a discrete $L^2$ norm defined by 
		$\norm[k]{v}^2 = \sum_{j = 0}^{N_k} k |v(x_j)|^2$ for functions $v$ on $(x_j)^{N_k}_{j=0}$.
		We have	
		\begin{equation*}
			\max_{n \{ 0, \ldots, N_k \}} \E[\norm[k]{X(t_n,\cdot) - X^{h,k}_D(t_n,\cdot)}^2]^{1/2} \le \max_{n,j \in \{0,\ldots,N_k\}} \E\left[ \left| X(t_n,x_j) - X^{h,k}_D(t_n,x_j) \right|^2 \right]^{1/2}.
		\end{equation*}
		Similarly, a (non-discrete) $L^2((0,D))$ bound can be obtained for $Y$ using~\eqref{prop:fem-convergence}. In that case, it is suboptimal for smooth noise, see the discussion after Remark~\ref{rem:sobolev-embedding}.
	\end{remark}
	
	\begin{figure}[ht!]
		\centering
		\includegraphics[width = \textwidth]{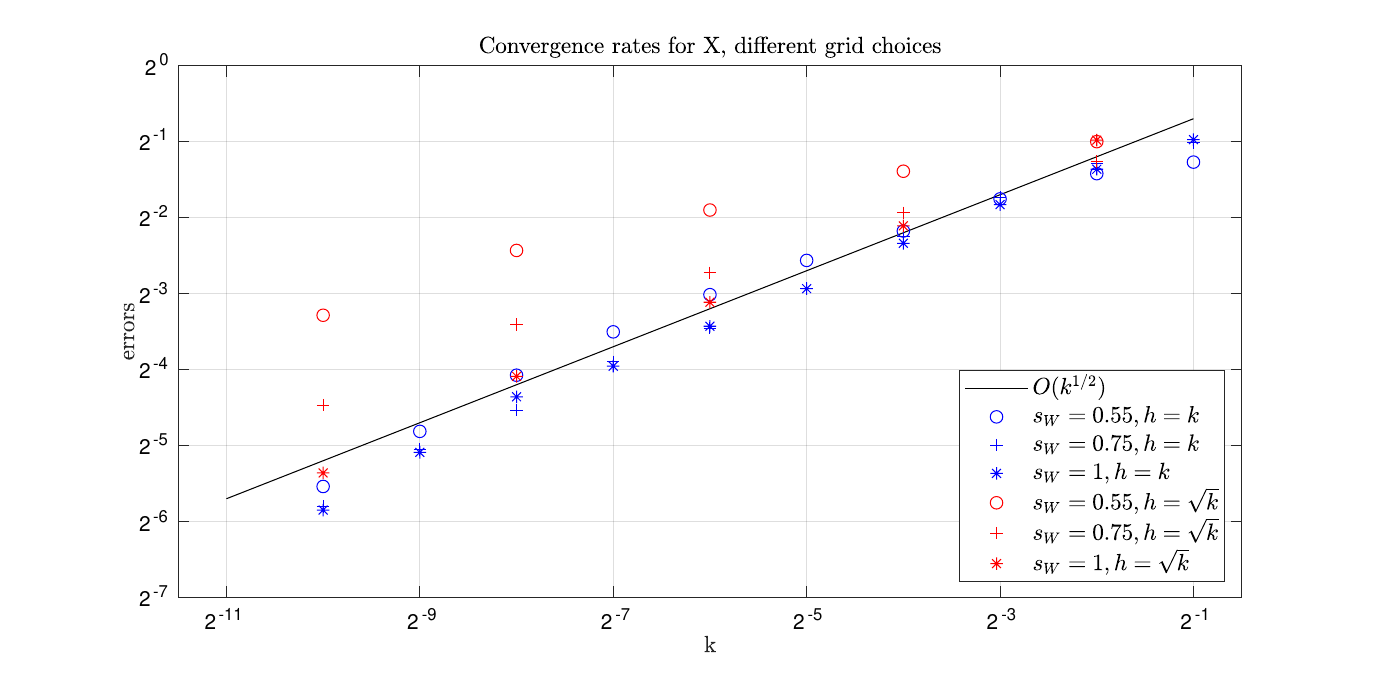}
		\caption{HEIDIH model errors $\max_{t_n,x_j}\E{[|X_D(t_n,x_j) - X^{k,h}_D(t_n,x_j)|^2]}^{1/2}$ for decreasing space and time mesh sizes. Here the approximation of $X$ is computed on a grid with equal space and time step sizes $k$, while $Y$ is computed with space step size $h$ equal to $k$ or $\sqrt{k}$.}
		\label{fig:temporalerror-comparison-plot}
	\end{figure}
	
	Suppose that $D$ is large, so that the error is dominated by the space and time discretization. The relationship between the spatial step size $h$ and time step size $k$ have important consequences for the efficiency of the algorithm. If $s_W$ is close to $1/2$, we must choose $h$ on the same order of magnitude as $k$ in order for the term $h^{s_2}$ not to dominate in the error estimate of Theorem~\ref{thm:full-approximation}. If $s_W \ge 1$, we can choose $h$ on the same order as $\sqrt{k}$ and still achieve a convergence rate of $1/2$ as $k \to 0$. We illustrate this point in Figure~\ref{fig:temporalerror-comparison-plot}, where we have let $T = 1$ and $a = 0.05$, and chosen $q_B$ and $\eta$ so that $\norm[H_{q_B}]{I^*_{H_{q_B} \hookrightarrow \cH^r}} = 1$. We have chosen $q_W$ as in Remark~\ref{rem:example-figure} with $q_s$ being a Matérn kernel with $\mu = \zeta = 1$. We let $\nu = s_W - 0.5$ range from $0.05$ to $0.5$ and plot the error 
	$$\max_{n,j \{ 0, \ldots, N_k \}}\E{[|X_D(t_n,x_j) - X^{k,h}_D(t_n,x_j)|^2]}^{1/2}$$
	for decreasing $k$ and choose either $h = k$ or $h = \sqrt{k}$. We have replaced $X_D(t_n,x_j)$ with a reference solution at $k = 2^{-11}$ in the first case, $k = 2^{-12}$ in the second, and used $N = 100$ samples in a Monte Carlo approximation of the error. We see that, as expected, the rate $1/2$ is achieved for all $s_W$ when $h = k$ but only for $s_W \ge 1$ when $h = \sqrt{k}$. If $s_W \ge 1$, it is important that the choice $h = \sqrt{k}$ is made, however, since this yields a much faster algorithm with the same convergence rate compared to $h = k$, see Figure~\ref{fig:comparison-times}.
	
	\begin{figure}[ht!]
		\centering
		\includegraphics[width = \textwidth]{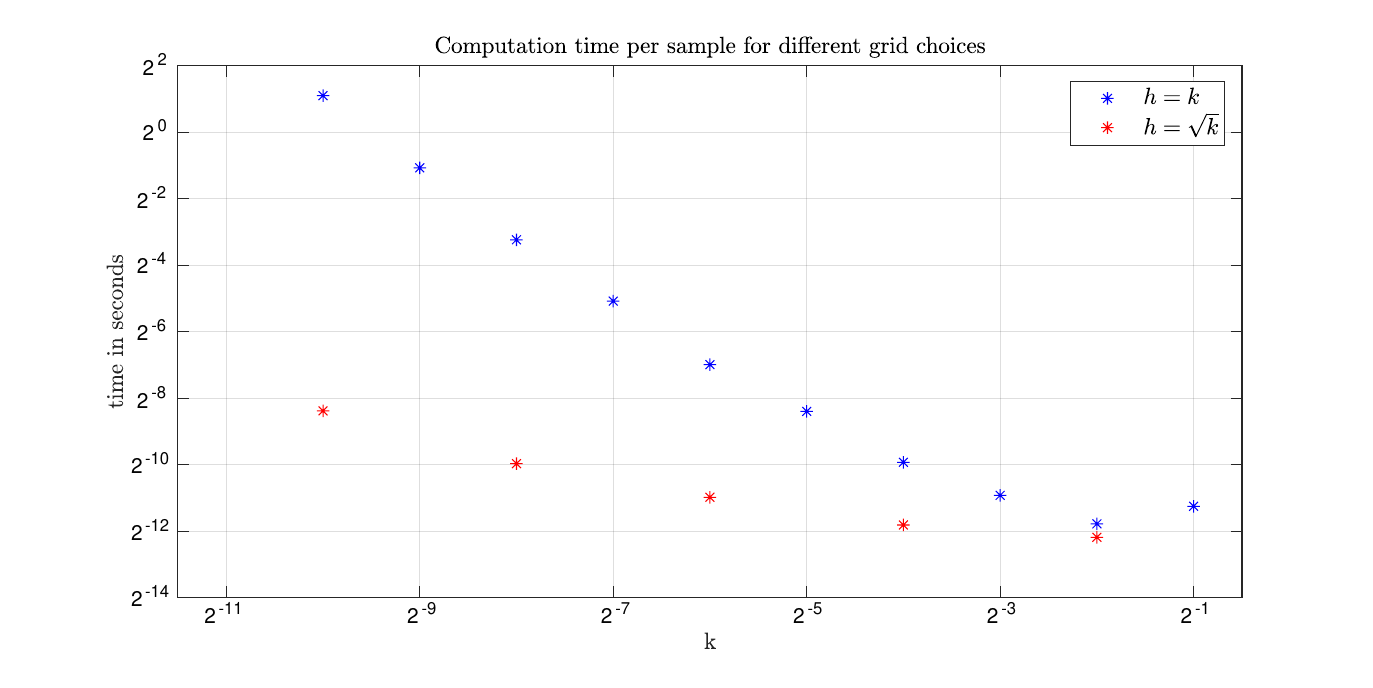}
		\caption{Computation time in seconds of one sample of $X^{k,h}_D$ for the different grid choices in Figure~\ref{fig:temporalerror-comparison-plot} for the case $s_W = 1$. The simulations were computed in MATLAB\textsuperscript{\tiny\textregistered} R2021b on an Intel\textsuperscript{\tiny\textregistered} Core\textsuperscript{\tiny\texttrademark} i7-10510U processor.}
		\label{fig:comparison-times}
	\end{figure}
	
	\section{Discussion and future research}
	\label{sec:conclusion}
	
	We end this paper with a brief discussion of our results with regard to novelty and future research. 
	
	The HEIDIH model was introduced as a special case of the infinite dimensional Heston stochastic volatility model of~\cite{BS18}. As far as we know, this paper is the first to consider numerical approximation of such a model, as well as providing a detailed regularity analysis. In particular, we focused on how regularity of the covariance kernels of the driving Wiener processes influence regularity of the model and gave concrete examples of covariance kernels that fit into our setting. A basic result for finite difference approximations that applies to all infinite dimensional Heston stochastic volatility models of forward prices was proven in Proposition~\ref{prop:X-error-decomp}. In the special case that $Z(t) = \eta$ for all $t \in [0,T]$ we were able to deduce and prove sharp convergence rates for a fully implementable and computationally efficient algorithm in Theorem~\ref{thm:full-approximation}. Future areas of research include the approximation of the HEIDIH model with other choices of $Z$, such as $Z(t) = Y(t)/\norm[\cH]{Y(t)}$ for $Y(t) \neq 0$. The main challenge then is to find an efficient way of computing the terms $\inpro[\cH^r]{\hat{Y}(t_i)}{\Delta B^i}$ and $\norm[\cH^r]{\hat{Y}(t_i)}$ in ~\eqref{eq:X-recursion}. Had $r = 0$ this could be solved by numerical integration, but for $r > 1/2$ the problem is harder. Another area for numerical development is other boundary conditions for $Y$. The Dirichlet condition was in our numerical simulations chosen due to the availability of suitable deterministic non-smooth data error estimates \cite{CLT94}. This led to relatively low convergence rates, as noted in Section~\ref{subsec:fem-section}.
	
	A large part of our numerical analysis was dedicated to finding sharp convergence rates for numerical approximations to the stochastic heat equation. There is a large body of literature dedicated to this topic, see~\cite{JK09} for a review of early results. In our setting, the situation is complicated by the fact that we need pointwise-in-space error estimates, that localization errors need to be taken into account and that sampling the Wiener process $W$ pointwise causes an additional error. 
	
	We are not aware of any pointwise-in-space error estimates for finite element approximations of the stochastic heat equation. However, for finite difference schemes there is a long tradition, starting with~\cite{G99} in the Walsh setting. The results of~\cite{G99} is for noise that is white in both space and time. We are not aware of any case for which noise colored by a concrete covariance kernel was considered for finite difference approximations. 
	
	Localization have been dealt with before in the Walsh setting, in~\cite{C22} for example. We showed how similar arguments can be adapted to the semigroup framework of \cite{DPZ14} considered in \cite{BS18}. 
	
	The question of how sampling the Wiener process $W$ pointwise affect the convergence of approximations of the stochastic heat equation has, as far as we know, barely been analyzed at all in the literature. In terms of finite elements it means replacing the $L_2$-projection $P^0_h$ with the interpolant $I_h$ in $P^0_h \Delta W^i = P^0_h (W(t_{i+1})- W(t_i))$ and the question of what effect this has is mentioned as open in~\cite[Section~6.2]{K14}. We showed in theory as well as by numerical simulations that there is a negative effect on the convergence if the noise is rough (i.e., if $H_{q_W} \cancel{\hookrightarrow} H^{s_W,\alpha}(\R)$ for $s_W \ge 1$) but not otherwise. After the submitting the first version of this paper, the result of replacing of $P_h^0$ with $I_h$ was analyzed in detail for a semi-linear SPDE with multiplicative noise on a bounded polygon \cite{LP22}.
	
	\bibliographystyle{abbrv}
	\bibliography{references}

\end{document}